\documentclass[ 11pt]{article}
\usepackage[top=0.9in, bottom=0.9in, left=0.9in, right=0.9in]{geometry}
\usepackage{amssymb}
\usepackage{amsthm}
\usepackage{amsfonts}
\usepackage{amsmath}
\usepackage{hyperref}
\usepackage{bbm}

\newtheorem{theorem}{Theorem}[section]
\newtheorem{lemma}[theorem]{Lemma}
\newtheorem{prop}[theorem]{Proposition}
\newtheorem{coro}[theorem]{Corollary}
\newtheorem{conj}[theorem]{Conjecture}

\theoremstyle{definition}

\newtheorem{remark}[theorem]{Remark}

\newcommand{\NN}{\mathbb{N}}

\newcommand{\Ec}{\mathcal{E}}

\newcommand{\setdef}{\ : \ }
\newcommand{\vep}{\varepsilon}
\newcommand{\psh}{{\rm PSH}}
\newcommand{\PSH}{{\rm PSH}}

\newcommand{\ddbar}{\partial\bar\partial}

\newcommand{\AM}{{I}}

\newcommand{\id}{\mathbbm{1}}

\newcommand{\dS}{d_{\mathcal {S}}}

\title{\vspace{-0.3in} The metric geometry of singularity types }

\author{Tam\'as Darvas, Eleonora Di Nezza, Chinh H. Lu}
\date{\vspace{-0.5cm}}

\begin{document}

\maketitle

\begin{abstract} Let $X$ be a compact K\"ahler manifold. Given a big cohomology class $\{\theta\}$, there is a natural equivalence relation on the space of $\theta$-psh functions giving rise to $\mathcal S(X,\theta)$, the space of singularity types of potentials. We introduce a natural pseudometric $d_{\mathcal {S}}$ on $\mathcal S(X,\theta)$ that is non-degenerate on the space of model singularity types and whose atoms are exactly the relative full mass classes. In the presence of positive mass we show that this metric space is complete. As applications, we show that solutions to a family of complex Monge-Amp\`ere equations with varying singularity type converge as governed by the $d_\mathcal S$-topology, and we obtain a semicontinuity result for multiplier ideal sheaves associated to singularity types, extending the scope of previous results from the local context.
\end{abstract}

\section{Introduction and main results}

Let $(X,\omega)$ be a K\"ahler manifold with a smooth closed $(1,1)$-form $\theta$. Let $\textup{PSH}(X,\theta)$ be the collection of integrable $\theta$-psh functions on $X$. With slight abuse of precision $u \in \textup{PSH}(X,\theta)$ if and only if $u$ is upper semi-continuous (usc), $u\in L^1(X,\omega^n)$,  and $\theta_u:=\theta + i\ddbar u \geq 0$ in the sense of currents. The set $\textup{PSH}(X,\theta)$ has plenty of members when $\{\theta\} \in H^2(X,\mathbb{C})$ is big, an assumption we will make throughout the paper.

Two potentials $u,v \in \textup{PSH}(X,\theta)$ have the same singularity type if and only if there exists $C \in \mathbb{R}$ such that $u - C \leq v \leq u + C$. This is easily seen to yield an equivalence relation, whose equivalence classes $[w],w \in \textup{PSH}(X,\theta)$ give rise to the space of singularity types $\mathcal S(X,\theta)$. This latter space plays an important role in transcendental algebraic geometry, as its elements represent the building blocks of multiplier ideal sheaves, log-canonical thresholds, etc., bridging the gap between the algebraic and the analytic viewpoint on the subject. We refer to the survey \cite{Dem15}  and references therein for insight into this ever expanding circle of ideas.

The space $\textup{PSH}(X,\theta)$ has a natural complete metric space structure given by the $L^1$ metric. However the $L^1$ metric does not naturally descend to $\mathcal S(X,\theta)$ making the study of variation of singularity type quite awkward and cumbersome. Indeed, reviewing the literature, ``convergence of singularity types'' is only discussed in an ad-hoc manner, under stringent conditions on the potentials involved. 

On the other hand, ``approximating" an arbitrary singularity type $[u]$ with one that is much nicer goes back to the beginnings of the subject. Perhaps the most popular of these approximation procedures is the one that uses Bergman kernels, as first advocated in this context by Demailly \cite{Dem92}. Here, using  Ohsawa-Takegoshi type theorems one obtains a (mostly decreasing) sequence  $[u_{j,B}]$ that in favorable circumstances approaches $[u]$ in the sense that multiplier ideal sheaves, log-canonical thresholds, vanishing theorems, intersection numbers etc. can be recovered in the limit (see for example \cite{Bo02, Bo04, DP04, Ca14, Dem15} and references therein). Still, no metric topology seems to be known that could  quantify the effectiveness or failure of the ``convergence" $[u_{j,B}] \to [u]$ (or that of other approximating sequences, for example the transcendental Bergman kernels suggested in \cite{Ber18}). In this work we propose an alternative remedy to this.

We introduce a natural (pseudo)metric $d_{\mathcal S}$ on ${\mathcal S}(X,\theta)$ and point out that it fits well with some already existing approaches in the literature. The precise definition of $d_\mathcal S$ uses the language of geodesic rays from \cite{DDL3,DL18} and is delayed until Section 3, however for the sake of a gentle introduction we note that there exists an absolute constant $C >1$ only dependent on $\dim_{\mathbb{C}}X$ such that:
$$d_\mathcal S([u],[v]) \leq   \sum_{j=0}^n \bigg(2\int_X \theta_{V_\theta}^j \wedge \theta_{\max(u,v)}^{n-j}-  \int_X \theta_{V_\theta}^j \wedge \theta_{v}^{n-j}-\int_X \theta_{V_\theta}^j \wedge \theta_{u}^{n-j}\bigg)\leq C  d_\mathcal S([u],[v]).$$
This is proved in Proposition \ref{prop: poor_Pythagorean_S}. Here $V_\theta$ is the least singular potential of $\textup{PSH}(X,\theta)$ and the integration is carried out over the respective non-pluripolar products introduced in \cite{BEGZ10}. Also, by \cite{WN19} we have that the expression in the middle is indeed non-negative, and as a result of $d_\mathcal S$ being a pseudo-metric, this expression will also satisfy the quasi-triangle inequality!

As we will see in Theorem \ref{criteria d_1 vanishing} below, $d_\mathcal S([u],[v])=0$ when the singularities of $u$ and $v$ are essentially indistinguishable (the Lelong numbers, multiplier ideal sheaves,  mixed masses of $[u]$ and $[v]$ are the same). More precisely, $d_\mathcal S([u],[v])=0$ if and only if $u$ and $v$ belong to the same relative full mass class, as introduced in \cite[Section 3]{DDL2}. In particular, $u \in \mathcal E(X,\theta)$ if and only if $d_\mathcal S([u],[V_\theta])=0$.
Consequently, the degeneracy of $d_\mathcal S$ is quite natural! 

Given the $d_\mathcal S$-continuity of $[u] \to \int_X \theta_u^n$ (Lemma \ref{lem: mixed_MA_dC_conv}) it is quite natural to introduce the following subspaces for any $\delta \geq 0$:
$$\mathcal S_\delta(X,\theta) : = \{[u] \in \mathcal S(X,\theta):  \ \int_X \theta_u^n \geq \delta \}.$$
These spaces are $d_\mathcal S$-closed, and according to our first main result they are also complete:
 
\begin{theorem}\label{thm: complete intro} For any $\delta > 0$ the space $(\mathcal S_\delta(X,\theta),d_\mathcal S)$ is complete.
\end{theorem}

Unfortunately the space  $(\mathcal S(X,\theta),d_\mathcal S)$ is not complete. This is quite natural however, as issues may arise if the non-pluripolar mass vanishes in the $d_\mathcal S$-limit (see Section  \ref{sect: incompleteness}, where we adapt an example of Demailly--Peternell--Schneider \cite{DPS94} to our context). 

As alluded to above, in general $L^1$-convergence of potentials (or even convergence in capacity) does not imply $d_\mathcal S$--convergence of their singularity types. However we note in Lemma \ref{lem: mon_limit_complete} below that if $u_j \nearrow u$ pointwise a.e. then $d_\mathcal S([u_j],[u]) \to 0$. In fact, Theorem \ref{thm: conv_subs_monotone} below gives a good intuition in general about what $d_\mathcal S$-convergence really means.
Omitting technicalities and somewhat abusing precision, this result shows that $d_\mathcal S([u_j],[u]) \to 0$ if and only if $u_j$ can be (subsequentially) sandwiched between two sequences of potentials $\psi_j \leq u_j \leq \chi_j$ such that $\{\psi_j\}_j$ is increasing, $\{\chi_j\}_j$ is decreasing and $\int_X \theta_{V_{\theta}}^l \wedge \theta^{n-l}_{\psi_j} \nearrow \int_X \theta_{V_{\theta}}^l \wedge \theta^{n-l}_{u}$ along with $\int_X \theta_{V_{\theta}}^l \wedge \theta^{n-l}_{\chi_j} \searrow \int_X \theta_{V_{\theta}}^l \wedge \theta^{n-l}_{u}$ for any $l \in \{0,\ldots,n-1\}$.

Suppose that $u,v \in \textup{PSH}(X,\theta)$ is such that $P(u,v):= \sup\{h \in \textup{PSH}(X,\theta) \ : \ h \leq \min(u,v)\} \in \textup{PSH}(X,\theta)$. Then $[\max(u,v)]$ and $[P(u,v)]$ represent the maximum and the minimum of the singularity types $[u],[v]$ respectively, and these four singularity types form a ``diamond" in the semi-lattice $\mathcal S(X,\theta)$. The following inequality between the masses of these potentials is of independent interest, and will be of great use in the proof of Theorem \ref{thm: conv_subs_monotone}  mentioned above.
\begin{theorem}\label{thm: volume_diamond_ineq_intr} Suppose that $u,v,P(u,v) \in \textup{PSH}(X,\theta)$. Then
$$\int_X \theta_u^n + \int_X \theta_v^n \leq \int_X \theta_{\max(u,v)}^n +\int_X \theta^n_{P(u,v)}.
$$
\end{theorem} 
As we will see, in case $\dim X = 1$, the above inequality is actually an identity, however strict inequality may occur if $\dim X \geq 2$ (see Remark \ref{rem: no equality in diamond}). 

\paragraph{Applications to multiplier ideal sheaves.}

For $[v] \in \mathcal S(X,\theta)$ we denote by $\mathcal J[v]$ the multiplier ideal sheaf associated to the singularity type $[v]$. Recall that $\mathcal J[v]$ is the sheaf of germs of holomorphic
functions $f$ such that $|f|^2 e^{-v}$ is locally integrable on $X$. Providing a positive answer to the Demailly strong openness conjecture \cite{DK01}, Guan--Zhou have shown that for any $u_j,u$ psh such that $u_j \nearrow u$ a.e.  we have that $\mathcal J[u_j]=\mathcal J[u]$ for $j \geq j_0$ \cite{GZh15,GZh16}, with a partial result obtained earlier by Berndtsson \cite{Bern15} (see also \cite{Dem15,Hiep14,Le17} for related results). Below we extend the scope of this theorem to the global context, providing a result that uses $d_\mathcal S$-convergence and avoids the condition $u_j \leq u$:

\begin{theorem}\label{thm: mult_ideal_semicont_intr} Let $[u],[u_j] \in \mathcal S(X,\theta), \ j \geq 0$, such that $d_\mathcal S([u_j],[u]) \to 0$.  Then there exists $j_0 \geq 0$ such that $\mathcal J[u] \subseteq \mathcal J[u_j]$ for all $j \geq j_0$.
\end{theorem}

The proof of this theorem involves an application of Theorem \ref{thm: volume_diamond_ineq_intr} and the local Guan--Zhou result for increasing sequences \cite{GZh15,GZh16}. Lastly, since $u_j \leq u$ trivially gives $\mathcal J[u_j] \subseteq \mathcal J[u]$, together with the $d_\mathcal S$-convergence criteria of Lemma \ref{lem: mon_limit_complete}, our theorem contains the global version of the Guan--Zhou result for increasing sequences of $\theta$-psh potentials. 

Motivated by a possible local analog of Theorem \ref{thm: mult_ideal_semicont_intr}  it would be interesting to see if a local version of the $d_\mathcal S$ metric exists on the space of singularity types of local psh potentials.

Note that equality in the inclusion $\mathcal J[u] \subseteq \mathcal J[u_j]$ of Theorem \ref{thm: mult_ideal_semicont_intr}  can not be expected in general. Indeed,  $d_\mathcal S([\lambda u],[u]) \to 0 $ as $\lambda \nearrow 1$ for any $u \leq 0$, however if $u$ has  log type singularity at some $x \in X$, but is locally bounded on $X \setminus \{x\}$, then $\mathcal J[u] \subsetneq\mathcal J[\lambda u] = \mathcal O_X, \lambda \in (0,1)$. 

\paragraph{Applications to variation of complex Monge--Amp\`ere equations.} Finally,  we turn to the application that motivated our introduction of the $d_\mathcal S$-topology. 

In a series of works \cite{DDL2,DDL3,DDL4} the authors studied solutions to equations of complex Monge--Amp\`ere type with prescribed singularity.
In a nutshell, one starts with a potential $\phi \in \textup{PSH}(X,\theta)$ and a density $0\leq f \in L^p(X)$, $p>1$, and is looking for a solution $\psi \in \textup{PSH}(X,\theta)$ such that $\theta_\psi^n = f \omega^n$ and $[\psi] = [\phi]$. By 
\cite{WN19} the condition $\int_X \theta_\phi^n = \int_X f \omega^n > 0$ is necessary for the solvability of this equation. Beyond this normalization condition, as it turns out, the necessary and sufficient condition for the well posedness is that $[\phi]$ satisfies $[\phi] = [P[\phi]]$, where
$$P[\phi] : = \sup\{v \in \textup{PSH}(X,\theta) \setdef [v] \leq [u], v \leq 0\}.$$
Singularity types $[\phi]$ satisfying the above condition are of \emph{model type}, and they appear in many natural contexts, as described in \cite{DDL2}.

One might ask the question, what happens if one considers a family of such equations, where the prescribed singularity type $[\phi_j]$ converges to some fixed singularity type $[\phi]$. In our next result we obtain that in such a case, the solutions $\psi_j$ converge to $\psi$ in capacity as expected, further evidencing the practicality of the $d_\mathcal S$-topology:
\begin{theorem} \label{thm: stability intro} Given $\delta >0$ and $p >1$ suppose that:\\
$\circ$ $[\phi_j],[\phi] \in \mathcal S_\delta(X,\omega), \ j \geq 0$ satisfy $[\phi_j] = [P[\phi_j]]$,  $[\phi] = [P[\phi]]$ and $d_\mathcal S([\phi_j],[\phi]) \to 0$. \\
$\circ$ $f_j,f \geq 0$ are such that $\| f\|_{L^p},\| f_j\|_{L^p}$, $p>1$, are uniformly bounded  and $f_j \to_{L^1}f$.\\
$\circ$ $\psi_j,\psi \in \textup{PSH}(X,\theta), \ j \geq 0$ satisfy $\sup_X \psi_j=0$, $\sup_X \psi=0$ and 
$$
\begin{cases}
\theta_{\psi_j}^n = f_j \omega^n\\
[\psi_j]=[\phi_j] \  
\end{cases},
\ \ \ 
\begin{cases}
\theta_{\psi}^n = f \omega^n\\
[\psi]=[\phi].
\end{cases}
$$
Then $\psi_j$ converges to $\psi$ in capacity, in particular $\|\psi_j - \psi \|_{L^1} \to 0$.
\end{theorem}
\paragraph{Organization.} In Section \ref{sect: preliminaries} we recall  several results in relative pluripotential theory developed recently by the authors. The metric $\dS$ along with its basic properties are introduced in Section \ref{sect: metric geo}. Theorem \ref{thm: complete intro} is proved in  Section \ref{sect: completeness} where an example is also given showing that the positive mass condition is necessary. Theorem \ref{thm: volume_diamond_ineq_intr} is proved in Section 5, Theorem \ref{thm: mult_ideal_semicont_intr} is proved in Section \ref{sect: ideal sheaves},  and Theorem \ref{thm: stability intro} is proved in Section \ref{sect: stability}. 

\paragraph{Acknowledgments.} The first named author has been partially supported by NSF grants DMS-1610202 and DMS-1846942(CAREER). This work was finished while the authors participated in the ``Research in Paris" program of Institut Henri Poincar\'e, and we would like to thank the institute for the hospitality and support.

\section{Preliminaries}\label{sect: preliminaries}

In this section we recall terminology and relevant results from the literature with focus on the works \cite{DDL1,DDL2,DDL3,DDL4}, as well as \cite{DL18}. We also point out some differences and extend the scope of some results whenever necessary.

\subsection{Model potentials and relative full mass classes}

Let $(X,\omega)$ be a compact K\"ahler manifold of dimension $n$ and fix $\theta$ a smooth closed $(1,1)$-form whose cohomology class is big. Our notation is taken from \cite{DDL2,DDL3,DDL4} and we refer to these works for further details.

A function $u: X \rightarrow \mathbb{R}\cup \{-\infty\}$ is called quasi-plurisubharmonic (quasi-psh) if locally $u= \rho + \varphi$, where $\rho$ is smooth and $\varphi$ is a plurisubharmonic (psh) function. We say that $u$ is $\theta$-plurisubharmonic ($\theta$-psh) if it is quasi-psh and $\theta_u:=\theta+i\ddbar u \geq 0$ in the weak sense of currents on $X$. We let $\PSH(X,\theta)$ denote the space of all $\theta$-psh functions on $X$ which are not identically $-\infty$. The class $\{\theta\}$ is {\it big} if there exists $\psi\in \psh(X,\theta)$ satisfying $\theta +i\ddbar \psi\geq \vep \omega$ for some $\vep>0$.   By the fundamental approximation theorem of Demailly \cite{Dem92}, if $\{\theta\}$ is big  there are plenty of $\theta$-psh functions.

 Given $u,v \in \textup{PSH}(X,\theta)$,  we say that 
\begin{itemize}\vspace{-0.1cm}
	\item $u$ is more singular than $v$, i.e., $u \preceq v$, if there exists $C\in \mathbb{R}$  such that $u\leq v+C$;\vspace{-0.1cm}
	\item $u$ has the same singularity as $v$, i.e., $u \simeq v$, if $u\preceq v$ and $v\preceq u$. \vspace{-0.1cm}
\end{itemize}
The classes $[u] \in \mathcal S(X,\theta)$ of this latter equivalence relation are called \emph{singularity types}. When $\theta$ is non-K\"ahler, all elements of $\textup{PSH}(X,\theta)$ are quite singular, and we distinguish the potential with the smallest singularity type in the following manner:
$$V_\theta := \sup \{u \in \textup{PSH}(X,\theta) \textup{ such that } u \leq 0\}.$$
A function $u\in \PSH(X,\theta)$ is said to have minimal singularity if it has the same singularity type as $V_{\theta}$, i.e., $[u]=[V_\theta]$.

Given $\theta^1,...,\theta^n$ smooth closed $(1,1)$-forms  and $\varphi_j \in \textup{PSH}(X,\theta^j)$, $j=1,...n$, following Bedford-Taylor \cite{BT76,BT82} in the local setting, it has been shown in \cite{BEGZ10} that the sequence of positive measures
\begin{equation}\label{eq: k_approx_measure}
\mathbbm{1}_{\bigcap_j\{\varphi_j>V_{\theta^j}-k\}}\theta^{1}_{\max(\varphi_1, V_{\theta^1}-k)}\wedge \ldots\wedge \theta^n_{\max(\varphi_n, V_{\theta^n}-k)}
\end{equation}
has total mass (uniformly) bounded from above and is non-decreasing in $k \in \Bbb R$, hence converges weakly to the so called \emph{non-pluripolar product} 
\[
\theta^1_{\varphi_1 } \wedge\ldots\wedge\theta^n_{\varphi_n }.
\]
The resulting positive measure does not charge pluripolar sets. In the particular case when $\varphi_1=\varphi_2=\ldots=\varphi_n=\varphi$ and $\theta^1=...=\theta^n=\theta$ we will call $\theta_{\varphi}^n$ the non-pluripolar Monge-Amp\`ere measure of $\varphi$, which generalizes the usual notion of volume form in case $\theta_{\varphi}$ is a smooth K\"ahler form.

An important property of the non-pluripolar product is that it is local with respect to the plurifine topology (see  \cite[Corollary 4.3]{BT87},\cite[Section 1.2]{BEGZ10}).  For convenience we record the following version for later use. 
\begin{lemma} \label{lem: plurifine}
Fix closed smooth big $(1,1)$-forms $\theta^1,...,\theta^n$.  Assume that $\varphi_j,\psi_j,j=1,...,n$ are $\theta^j$-psh functions such that $\varphi_j =\psi_j$ on $U$ an open set in the plurifine topology. Then 
$$
\mathbbm{1}_{U} \theta^1_{\varphi_1} \wedge ... \wedge \theta^n_{\varphi_n} = \mathbbm{1}_{U} \theta^1_{\psi_1} \wedge ... \wedge \theta^n_{\psi_n}.
$$
\end{lemma}
Lemma \ref{lem: plurifine} will be referred to as the plurifine locality. For practice we note that sets of the form $\{u<v\}$, where $u,v$ are quasi-psh functions, are open in the plurifine topology.

As a consequence of Bedford-Taylor theory, the measures in \eqref{eq: k_approx_measure} all have total mass less than $\int_X \theta_{V_\theta}^n$, in particular, after letting $k \to \infty$ we notice that $\int_X \theta_{\varphi}^n \leq \int_X \theta_{V_\theta}^n$. In fact it was proved in \cite[Theorem 1.2]{WN19} that for any $u,v \in \textup{PSH}(X,\theta)$  the following monotonicity result holds for the masses:
$$v \preceq u \Longrightarrow \int_X \theta_v^n \leq \int_X \theta_u^n.$$

This result was extended in \cite{DDL2} for non-pluripolar products building on the following fundamental convergence property. 

\begin{theorem}
	\label{thm: lsc of non pluripolar product}
	Let $\theta^j, j \in \{1,\ldots,n\}$ be smooth closed $(1,1)$-forms on $X$ whose cohomology classes are big. Suppose that for all $j \in \{1,\ldots,n\}$  we have $u_j,u_j^k\in \textup{PSH}(X,\theta^j)$ such that  $u^k_j \to u_j$ in capacity as $k \to \infty$. If $\chi_k\geq 0$ is a sequence of uniformly bounded  quasi-continuous functions which converges in capacity to a quasi-continuous function $\chi\geq 0$, then
	\begin{equation}\label{eq: key convergence}
	\liminf_{k\to +\infty} \int_X \chi_k \theta^1_{u^k_1} \wedge \ldots \wedge \theta^n_{u^k_n}  \geq  \int_X \chi  \theta^1_{u_1} \wedge \ldots \wedge \theta^n_{u_n}. 
	\end{equation}
If additionally,  
	\begin{equation}\label{eq: global_mass_semi_cont}
	\int_ X \theta^1_{u_1} \wedge \ldots \wedge \theta^n_{u_n} \geq \limsup_{k\rightarrow \infty} \int_X \theta^1_{u^k_1} \wedge \ldots \wedge \theta^n_{u^k_n},
	\end{equation}
then $\theta^1_{u^k_1} \wedge \ldots \wedge \theta^n_{u^k_n}$ weakly converges  to  $\theta^1_{u_1} \wedge \ldots \wedge \theta^n_{u_n}$.
\end{theorem}

Note that this result is slightly more general than \cite[Theorem 2.3]{DDL2} but the proof is the same.  
Shadowing Bedford--Taylor theory \cite{BT82,BT87}, the above convergence and monotonicity results opened the door to the development of relative finite energy pluripotential theory, whose terminology we now partially recall from \cite[Sections 2-3]{DDL2}. 

\paragraph*{The relative full mass classes $\Ec(X,\theta,\phi)$.}
Fixing $\phi \in \textup{PSH}(X,\theta)$ one can consider only $\theta$-psh functions that are more singular than $\phi$. Such potentials form the set $\textup{PSH}(X,\theta,\phi)$. Since the map $[u] \to \int_X \theta_u^n$ is monotone increasing, but not strictly increasing, it is natural to consider the set of $\phi$-relative \emph{full mass potentials}:
$$\mathcal E(X,\theta,\phi) := \left\{u \in \textup{PSH}(X,\theta,\phi) \ \textup{ such that } \int_X \theta_u^n = \int_X \theta_\phi^n\right \}.$$
Naturally, when $v \in \textup{PSH}(X,\theta,\phi)$ we only have $\int_X \theta^n_v \leq \int_X \theta^n_\phi$. As pointed out in \cite{DDL2,DDL4}, when studying the potential theory of the above space, the following well known envelope constructions are of great help:
$$ P_\theta(\psi,\chi), \ P_\theta[\psi](\chi),  \ P_\theta[\psi] \in \textup{PSH}(X,\theta).
$$

In the context of K\"ahler geometry these were introduced by Ross and Witt Nystr\"om \cite{RWN14}, using slightly different notation. Given any $f : X \to \Bbb [-\infty,+\infty]$ the starting point is the envelope $P_\theta(f):=\textup{usc}(\sup\{v \in \textup{PSH}(X,\theta), \ v \leq f \})$. Then, for $\psi,\chi \in \textup{PSH}(X,\theta)$ we can introduce the ``rooftop envelope'' $P_\theta(\psi,\chi):=P_\theta(\min(\psi,\chi))$. This allows us to further introduce
$$P_\theta[\psi](\chi) := \textup{usc}\Big(\lim_{C \to +\infty}P_\theta(\psi+C,\chi)\Big).$$

It is easy to see that $P_\theta[\psi](\chi)$ depends on the singularity type $[\psi]$. When $\chi = V_\theta$, we will simply write $P[\psi]:=P_\theta[\psi]:=P_\theta[\psi](V_\theta)$ and call this potential the \emph{envelope of the singularity type} $[\psi]$. It follows from \cite[Theorem 3.8]{DDL2}, \cite{Ber18}, \cite{GLZ17} that $\theta_{P[\psi]}^n \leq \mathbbm{1}_{\{P[\psi] =0\}} \theta^n$. 
Also, by \cite[Proposition 2.3 and Remark 2.5]{DDL3} we have that $\int_X \theta_{P[\psi]}^n = \int_X \theta_\psi^n$.

Using such envelopes, in \cite[Theorem 1.3]{DDL2} we characterized membership in $\mathcal E(X,\theta,\phi)$:

 \begin{theorem}\label{thm: DDL2_E_char} Suppose $\phi \in \textup{PSH}(X,\theta)$  and $\int_X \theta_\phi^n >0$. Then $u \in \mathcal E(X,\theta,\phi)$ if and only if $u \in \textup{PSH}(X,\theta,\phi)$ and $P[u]=P[\phi]$.
\end{theorem}

For further results about the connection of envelopes and relative full mass classes we refer to \cite[Section 3]{DDL2}.

\paragraph{The ceiling operator and model potentials.} We consider the \emph{ceiling} operator $\mathcal C:\textup{PSH}(X,\theta) \to \textup{PSH}(X,\theta)$ defined by
$$
\mathcal C(u) := \textup{usc} (\sup \mathcal F_u),
$$
where 
\begin{equation}
    \mathcal F_u:= \left \{v \in \textup{PSH}(X,\theta) \setdef \ [u] \leq [v], \ v \leq 0, \ \int_X \theta_v^k\wedge \theta_{V_\theta}^{n-k} = \int_X  \theta_u^k\wedge \theta_{V_\theta}^{n-k} , \  k\in \{0,...,n\} \right \}. \label{eq: Fu}
\end{equation}
As it turns out, there is no reason to take the upper semi-continuous regularization in the definition above, as $\mathcal C(u)$ is a candidate in its defining family $\mathcal F_u$. This is confirmed by the next lemma.

\begin{lemma}\label{lem: ceiling as limit}
Assume that $u\in \psh(X,\theta)$ and  $u\leq 0$. Then
\begin{equation}\label{eq: ceiling_envelope_id}
\mathcal{C}(u) = \lim_{\varepsilon\to 0^+} P{[(1-\varepsilon)u+\varepsilon V_\theta]} \in \mathcal{F}_u. 
\end{equation}
In particular, if $\phi,\psi \in \textup{PSH}(X,\theta)$ with $[\phi] \leq [\psi]$ then $\mathcal C(\phi) \leq \mathcal C(\psi)$, i.e., $\mathcal C$ is monotone increasing.

\end{lemma}
\begin{proof}We have that $\lim_{\varepsilon\to 0^+} ((1-\varepsilon)u+\varepsilon V_\theta)= u$ and that $u_{\varepsilon}:=P{[(1-\varepsilon)u+ \varepsilon V_\theta]}\geq P{[u]}\geq  u$ is decreasing as $\varepsilon\rightarrow 0+$. We denote by $u_0$ its limit.
Next, we claim that $u_0\in \mathcal{F}_u$. First, observe that since $u_0\geq u$, by \cite[Theorem 1.1]{DDL2} we have that, fixing $k \in \{1,2,...,n\}$,
$$\int_X \theta_{u_0}^k \wedge \theta_{V_\theta}^{n-k}\geq \int_X \theta_{u}^k \wedge \theta_{V_\theta}^{n-k}.$$ Moreover, it follows from the multilinearity of the non-pluripolar product that for all $\varepsilon >0$, 
\begin{eqnarray*}
\int_X \theta_{u_0}^k \wedge \theta_{V_\theta}^{n-k}  &\leq & \int_X \theta_{u_\varepsilon}^k \wedge \theta_{V_\theta}^{n-k} 
 =  \int_X \theta_{(1-\varepsilon)u+\varepsilon V_\theta }^k \wedge \theta_{V_\theta}^{n-k}
 =  (1-\varepsilon)^k\int_X \theta_{u}^k \wedge \theta_{V_\theta}^{n-k} + O(\varepsilon),
\end{eqnarray*}
where in the first inequality we used \cite[Theorem 1.1]{DDL2}, in the first equality we used \cite[Proposition 2.1, Theorem 2.2] {DDL2} and the definition of $P{[(1-\varepsilon)u+\varepsilon V_\theta]}$, and in the second equality we used the multilinearity of the non-pluripolar product. Letting $\varepsilon\to 0$ we prove the claim, hence $u_0\leq \mathcal{C}(u)$. 

On the other hand, as $(1-\varepsilon)u+\varepsilon V_\theta$ satisfies the non-collapsing condition $\int_X \theta_{(1-\varepsilon)u+\varepsilon V_\theta}^n \geq \varepsilon^n \int_X \theta_{V_\theta}^n>0$, it follows from \cite[Remark 2.5]{DDL2} and \cite[Theorem 3.12]{DDL2} that $P{[(1-\varepsilon)u+\varepsilon V_\theta]}$ is the maximal element of $\mathcal{F}_{(1-\varepsilon)u+\varepsilon V_\theta}$. Due to multilinearity of non-pluripolar products, it follows from the above definition of $\mathcal{F}_u$ that 
\begin{equation}\label{eq: scaled_inclusion}
(1-\varepsilon) \mathcal{F}_u + \varepsilon V_\theta \subset \mathcal{F}_{(1-\varepsilon)u+\varepsilon V_\theta}.
\end{equation}
This implies that $(1-\varepsilon)\mathcal{C}(u) +\varepsilon V_\theta\leq P{[(1-\varepsilon)u+\varepsilon V_\theta]}$. Hence letting $\varepsilon\to 0^+$ we obtain $\mathcal{C}(u)\leq u_0$, proving the first statement. The last statement follows from \eqref{eq: ceiling_envelope_id} together with the fact that if $[\phi]\leq [\psi]$ then $P{[(1-\varepsilon)\phi+\varepsilon V_\theta]}\leq  P{[(1-\varepsilon)\psi+\varepsilon V_\theta]}$. 
\end{proof}

In this work, we say that a potential $\phi \in \textup{PSH}(X,\theta)$ is a \emph{model} potential if $\phi = \mathcal C(\phi)$, i.e., if $\phi$ is a fixed point of $\mathcal C$. Similarly, the corresponding singularity types $[\phi]$ are called model type singularities. We note that this definition is seemingly different from the one in \cite{DDL2,DDL3,DDL4}, where we said that $\phi$ is model in case $\phi = P[\phi]$! Thankfully, this inconsistency will cause little to no disruption: in the important particular case of non-vanishing mass, i.e. $\int_X \theta_{\phi}^n > 0$,  \cite[Remark 2.5, Theorem 3.12]{DDL2} gives that $P[\phi]=\mathcal C(\phi)$, hence these two definitions are indeed the same. We predict that this is the case in general as well:
\begin{conj} For any $\phi \in \textup{PSH}(X,\theta)$ we have that $P[\phi]=\mathcal C(\phi)$.
\end{conj}

To finish this paragraph we list and prove a number of  basic properties of the ceiling operator:

\begin{prop} \label{prop: C_operator_properties} Suppose $v \in \textup{PSH}(X,\theta)$. The following hold:\\
\noindent (i) if $\int_X \theta_v^n>0$ then $\mathcal C(v) = P[v]$.\\
\noindent (ii) $\lim_{\varepsilon \rightarrow 0} \mathcal C((1-\varepsilon)v+\varepsilon V_\theta) =\lim_{\varepsilon \rightarrow 0} P{[(1-\varepsilon) v+\varepsilon V_\theta]} = \mathcal C(v)$.\\
\noindent (iii) $\mathcal C(\mathcal C(v))=\mathcal C(v)$.\\
\noindent (iv) $\mathcal C(P[v])=\mathcal C(v)$.
\end{prop}

\begin{proof}
The first property is a consequence of \cite[Theorem 2.5, Theorem 3.12]{DDL2}. 
The statement in $(ii)$ follows from Lemma \ref{lem: ceiling as limit} together with $(i)$. To prove $(iii)$ we note that $\mathcal F_{\mathcal C(v)} \subset \mathcal F_v$ implies that $\mathcal C(\mathcal C(v)) \leq \mathcal C(v)$ (via Lemma \ref{lem: ceiling as limit}). Since $v \leq \mathcal C(v)$, the other inequality in (iii) follows from monotonicity of $\mathcal C$.
Note that \cite[Remark 2.5]{DDL2} implies that $P[v] \in \mathcal F_v$, hence  $v \leq P[v]\leq \mathcal C(v)$. Applying $\mathcal C$ to these inequalities together with $(iii)$ gives $(iv)$.
\end{proof}

\paragraph*{The Monge--Amp\`ere energy $I$ and the finite energy class $\mathcal E^1(X,\theta)$.}
As further evidenced by the next lemma, potentials with model type singularity play a distinguished role in the theory (see  \cite[Lemma 2.2]{DDL4} for a more precise result): 

\begin{lemma}\label{lem: compactness and model type}
Let $\phi \in \PSH(X,\theta)$ and $\phi= P[\phi]$. Then, for all $u\in \PSH(X,\theta,\phi)$  we have $\sup_X (u-\phi) =\sup_X u$, and  the set 
$$\mathcal{F}:=\{u \in \PSH(X,\theta,\phi) \setdef \sup_X (u-\phi)=0\}$$ 
is relatively compact in the $L^1$-topology of potentials.
\end{lemma}
\begin{proof} Let $u \in \textup{PSH}(X,\theta, \phi)$. Since $\phi \leq 0$ it follows that $\sup_X (u - \phi) \geq \sup_X u$. For the other direction we notice that $u - \sup_X u \leq  \phi$, hence $\sup_X (u-\phi) \leq \sup_X u$. Relative compactness of $\mathcal F$ then follows from \cite[Proposition 8.5]{GZ17}.
\end{proof}

We define the Monge-Amp\`ere energy of any $\theta$-psh function $u$ with minimal singularities as 
$$
\mathrm{I} (u) :=\frac{1}{n+1} \sum_{k=0}^n \int_X (u-V_\theta) \theta_u^k \wedge\theta_{V_\theta}^{n-k}. 
$$
We then define the Monge-Amp\`ere energy for arbitrary $u\in \mathrm{PSH}(X,\theta)$ as
$$
I(u) : =\inf \{I(v) \setdef  v\in \mathrm{PSH} (X,\theta), \; v\ \textrm{has minimal singularities, and } u\leq v\}. 
$$
We let $\mathcal{E}^1(X,\theta)$ denote the set of all $u\in \psh(X,\theta)$  such that $I(u)$ is finite. Since $\theta$ will be fixed throughout the paper we will occasionally denote this space simply as $\mathcal{E}^1$.
In the next theorem we collect basic properties of the Monge-Amp\`ere energy:

\medskip 

\begin{theorem}\label{thm: basic I energy} Suppose $u,v \in \mathcal{E}^1(X,\theta)$. The following hold:\\
\noindent (i) $\AM(u)-\mathrm{I}(v) = \frac{1}{n+1}\sum_{k=0}^n \int_X (u-v) \theta_{u}^k \wedge \theta_{v}^{n-k}.$\\
\noindent (ii) If $u\leq V_\theta$ then, $
\int_X (u-V_\theta) \theta_u^n \leq \AM(u) \leq \frac{1}{n+1} \int_X (u-V_\theta) \theta_{u}^n. $ \\
\noindent (iii) $\AM$ is non-decreasing and concave along affine curves. Additionally, the following estimates hold: $
	\int_X (u-v) \theta_u^n \leq \AM(u) -\AM(v) \leq \int_X (u-v) \theta_v^n.$
\end{theorem}
\noindent In particular, we observe that $u\leq v$ implies $I(u)\leq I(v)$. We refer to \cite[Theorem 2.1 and Proposition 2.2]{DDL3} for a proof.

\begin{lemma}\label{lem: energy equal}
Let $u, v\in \mathcal{E}^1(X,\theta)$ be such that $u\geq v$ and $I(u)=I(v)$. Then $u=v$.
\end{lemma}
\begin{proof}
Since $I(u)=I(v)$, and $u\geq v$, Theorem \ref{thm: basic I energy}$(i)$ implies that $\int_X (u-v)\theta_v^n= 0$.  We then have  $\theta_v^n(u>v)=0$, hence the domination principle \cite[Proposition 2.4]{DDL1}  gives $u=v$. 
\end{proof}

We recall that, given $u,v \in \mathcal E^1(X,\theta)$, it has been shown in \cite[Theorem 2.10]{DDL1} that $P(u,v)$ belongs to $\Ec^1(X,\theta)$ as well. As in \cite[Section 3]{DDL3} we define:
\begin{equation}\label{def d1}
d_1(u,v)=\AM(u) + \AM(v) - 2\AM(P(u,v)).
\end{equation}
By \cite[Theorems 1.1]{DDL3} the space $(\mathcal E^1(X,\theta),d_1)$ is a complete geodesic metric space whose geodesic segments arise as $d_1$-limits to solutions to a degenerate complex Monge-Amp\`ere equation (they are sometimes referred to as finite energy geodesics). Also, the Monge-Amp\`ere energy is linear along these geodesics.
In the following we adapt some results of \cite{BDL15} and \cite{Dar15} to the big setting.

\begin{prop}\label{prop 5.1: adaptation BDL}
Suppose $[0, 1] \ni t \rightarrow  u_t , v_t\in \mathcal{E}^1$ are finite energy geodesics. Then the map $t \rightarrow \AM(P (u_t, v_t))$ is concave. Consequently, the map $t \rightarrow d_1(u_t, v_t)$ is convex.
\end{prop}

\begin{proof}
The result follows using the same arguments as \cite[Proposition 5.1]{BDL15} together with \cite[Theorems 2.10, 3.12 and Proposition 3.2]{DDL1}.
\end{proof}
Next we point out that the $d_1$-geodesics are ``endpoint stable'':
\begin{prop}\label{prop 4.3: adaptation BDL}
Let $[0,1]\ni t\rightarrow u_t^j\in \mathcal{E}^1$ be a sequence of finite energy geodesic segments such that $d_1(u_0^j, u_0), d_1(u_1^j, u_1)\rightarrow 0$. Then $d_1(u_t^j, u_t)\rightarrow 0$ for all $t$, where $[0, 1] \ni t \rightarrow  u_t \in \mathcal{E}^1$ is the finite energy geodesic segment connecting $u_0, u_1$.
\end{prop}
\begin{proof}Exactly the same as \cite[Proposition 4.3]{BDL15}.
\end{proof}
\begin{prop}\label{lemma 4.15: adaptation Dar15}
Given $u, v \in \mathcal{E}^1$, we have $|\AM(u) -\AM(v)| \leq d_1(u, v)$.
\end{prop}
\begin{proof}
We observe that by the monotonicity of the energy (Theorem \ref{thm: basic I energy}(i)) and since $P(u,v)\leq \min(u, v)$ we have $\AM(P(u,v))\leq \min(\AM(u), \AM(v))$. The result then follows directly from the definition of $d_1$.
\end{proof}

Lastly, we point out  the analog of \cite[Remark 5.6]{Dar15}:

\begin{prop}\label{remark 5.6: adaptation Dar15}
Given $u,v \in \mathcal{E}^1$, there exists $C =C(n) > 1$ such that
$$ d_1(u,v) \leq d_1 (u,\max(u,v)) + d_1(\max(u,v),v)\leq C d_1(u,v). $$
\end{prop}
\begin{proof}
The first inequality follows from the triangle inequality. The second inequality follows from \cite[Theorem 3.7 and eq. (6)]{DDL3}.
\end{proof}

\subsection{The chordal  geometry of geodesic rays}\label{sec: chordal geometry}
By $\mathcal R(X,\theta)$ we denote the space of finite energy geodesic rays emanating from $V_\theta$:
$$\mathcal R(X,\theta):= \{[0,\infty) \ni t \to u_t \in \mathcal E^1(X,\theta) \textup{ s.t. } u_0 =V_\theta  \textup{ and } t \to u_t \textup{ is a $d_1$ geodesic ray}\}.$$  
As shorthand convention we will use the notation $\{u_t\}_t \in \mathcal R(X,\theta)$ when referring to rays.

According to the constructions in \cite{DDL3} the above space has plenty of elements. In \cite[Section 4]{DL18} the first and last author carried out a detailed analysis of the space of $L^p$ geodesic rays in the K\"ahler case. For similar flavour results in the non-Archimedean context we refer to \cite{BBJ18} and references therein.

Given that our main focus here is on the space of singularity types, in the present paper we only focus on the basic analysis of the space of $L^1$ rays in the more general case of big cohomology classes.

In case of the Euclidean topology of $\Bbb R^n$, the space of (unit speed) geodesic rays emanating from the origin, is just the collection of half lines emanating from $0$, that can be identified with the unit sphere. Inspired by this simple analogy and the chordal metric structure on the sphere, we introduce the chordal $L^1$  geometry on $\mathcal R(X,\theta)$:
\begin{equation}\label{eq: d1 C def}
d_1^c(\{u_t\}_t,\{v_t\}_t) := \lim_{t \to \infty}\frac{d_1(u_t,v_t)}{t}.
\end{equation}
Note that by Proposition \ref{prop 5.1: adaptation BDL} it follows that $t \to d_1(u_t,v_t)$ is convex, hence the map $t \to d_1(u_t,v_t)/t$ is increasing, implying that the limit in \eqref{eq: d1 C def} is well defined. We also note the following theorem:

\begin{theorem}\label{thm: R_complete} The space $(\mathcal R(X,\theta),d_1^c)$ is a complete metric space.
\end{theorem}
\begin{proof} 

The triangle inequality is inherited from the $d_1$-geometry of $\mathcal E^1(X,\theta)$. To show that $d_1^c$ is non-degenerate, suppose $d_1^c(\{u_t\}_t,\{v_t\}_t)=0$ and fix $t_0 >0$.  By Proposition \ref{prop 5.1: adaptation BDL} we have:
\begin{equation}\label{eq: fixed_time_cov}
\frac{d_1(u_{t_0},v_{t_0})}{t_0} \leq \lim_{t \to \infty} \frac{d_1(u_t,v_t)}{t}=0.
\end{equation}
Since $d_1$ is non-degenerate, we obtain that $u_{t_0}=v_{t_0}$ and since $t_0>0$ was arbitrary, we obtain that the geodesics $t \to u_t$ and $t \to v_t$ are the same.

Now we argue that $(\mathcal R(X,\theta),d_1^c)$ is complete. Let $\{t \to u_t^k \}_k$ be a $d_1^c$-Cauchy sequence. For fixed $t_0 >0$ by \eqref{eq: fixed_time_cov} we deduce that $\{u_{t_0}^k\}_k$ is a $d_1$-Cauchy sequence in $(\mathcal E^1(X,\theta),d_1)$. By the completeness of the latter space there exists $u_{t_0} \in \mathcal E^1(X,\theta)$ such that $d_1(u_{t_0}^k,u_{t_0}) \to 0$ as $k\rightarrow +\infty$. As a result we obtain a ``limit curve" $[0,\infty) \ni t \to u_t \in \mathcal E^1(X,\theta)$. By Proposition \ref{prop 4.3: adaptation BDL} we obtain that $t \to u_t$ is in fact a finite energy geodesic ray emanating from $u_0 =V_\theta$. 

To finish the proof, we have to argue that $d_1^c(\{u_t^k\},\{u_t\}) \to 0$. Fix $\varepsilon >0$. By \eqref{eq: fixed_time_cov} there exists $j_\varepsilon >0$ such that $d_1(\{u^k_t\},\{u^{l}_t\}) < \varepsilon t$ for all $t >0$ and $k,l > j_\varepsilon$.  Letting $k \to \infty$, we obtain that $d_1(\{u_t\},\{u^{l}_t\}) < \varepsilon t$, hence $d_1^c(\{u_t\}_t,\{u^l_t\}_t) < \varepsilon$ for all $l > j_\varepsilon$, finishing the proof. 
\end{proof}

\paragraph{The radial Monge-Amp\`ere energy of $\mathcal R(X,\theta)$.} For $\{u_t\}_t \in \mathcal R(X,\theta)$ it is natural to introduce the radial Monge--Amp\`ere energy $I\{\cdot\} : \mathcal R(X,\theta) \to \Bbb R$ by the formula
$I\{u_t\} = \lim_t \frac{I(u_t)}{t} = I(u_1)$. By the $d_1$-Lipschitz property of $I$ on $\mathcal E^1(X,\theta)$ (Proposition \ref{lemma 4.15: adaptation Dar15}) and the fact that the map $t \to d_1(u_t,v_t)/t$ is increasing it follows that:
\begin{equation}\label{eq: I_R_cont}
\big|I\{u_t\} - I\{v_t\}\big| \leq d^c_1(\{u_t\}_t,\{v_t\}_t), \ \ \{u_t\}_t,\{v_t\}_t \in \mathcal R(X,\theta).  
\end{equation}

\paragraph{The metric decomposition inequality of $\mathcal R(X,\theta)$.} 

The analog of the Pythagorean formula holds for the space $(\mathcal R(X,\theta),d_1^c)$ (see \cite[Example 3.2]{Xia19} for the argument in the K\"ahler case that translates easily to our context as well). However this formula  does not descend to $\mathcal S(X,\theta)$. Out of this reason, we derive the radial analog of Proposition \ref{remark 5.6: adaptation Dar15} instead. Though perhaps not as ``flashy" as the Pythagorean formula, this ``decomposition inequality" has a number of similar consequences, and it also descends to $\mathcal S(X,\theta)$ as well (see Section \ref{sect: metric geo} below).

First we need to define what we understand under the maximum of two geodesic rays $\{u_t\}_t,\{v_t\}_t \in \mathcal R(X,\theta)$. This is simply the smallest ray $\{h_t\}_t \in \mathcal R(X,\theta)$ that lies above $\{u_t\}_t,\{v_t\}_t$. It is elementary to see that such a ray does exist. Indeed, $h_t = {\rm usc}(\lim_{l \to \infty} w^l_t)$, where $[0,l] \ni t \to w^l_t \in \mathcal E^1(X,\theta)$ is the finite energy geodesic segment joining $w^l_0 = V_\theta$ and $w^l_l = \max(u_l,v_l)$. Since $\{\max(u_t,v_t)\}_t$ is a subgeodesic ray, by the comparison principle \cite[Proposition 3.3]{DDL1} it can be seen that each sequence $\{w^l_t \}_l$ is increasing, proving that $\{h_t\}_t$ is indeed a geodesic ray. By construction this ray has to be the smallest ray lying above $\{u_t\}_t,\{v_t\}_t$, hence it makes sense to introduce the notation:
$$\textup{max}_\mathcal R(u_t,v_t) := h_t.$$
Next we show that the rays $u_t,v_t,\textup{max}_\mathcal R(u_t,v_t)$ satisfy a ``metric decomposition inequality":
\begin{prop} \label{prop: poor_Pythagorean_R}There exists $C>1$ such that, for all $\{u_t\}_t,\{v_t\}_t \in \mathcal R(X,\theta)$,
\begin{equation}
 d_1^c(\{u_t\}_t,\{v_t\}_t) \leq d_1^c(\{u_t\}_t,\{\textup{max}_\mathcal R(u_t,v_t)\}_t) + d_1^c(\{\textup{max}_\mathcal R(u_t,v_t)\}_t,\{v_t\}_t)\leq C  d_1^c(\{u_t\}_t,\{v_t\}_t).
\end{equation}
\end{prop}
\begin{proof} The first estimate follows from the triangle inequality. 
Since $\textup{max}_\mathcal R(u_t,v_t) \geq u_t$, from \eqref{def d1} we have that 
\begin{equation}\label{eq: d_1_est_begin}
d_1^c(\{u_t\}_t,\{\textup{max}_\mathcal R(u_t,v_t)\})=\lim_{t\to \infty }\frac{I(\textup{max}_\mathcal R(u_t,v_t)) - I(u_t)}{t}=I(\textup{max}_\mathcal R(u_1,v_1)) - I(u_1),
\end{equation}
where the last identity follows from the linearity of $I$ along geodesic rays \cite[Theorem 3.12]{DDL1}.

By construction of the ray $t \to \textup{max}_\mathcal R(u_t,v_t)$ it follows that
$$I(\textup{max}_{\mathcal R}(u_1,v_1)) = \lim_{l \to \infty} \frac{I(\max(u_l,v_l))}{l},$$
and this last limit exists as $l \to I(\max(u_l,v_l))$ is convex thanks to \cite[Theorem 3.8]{DDL1}. Consequently, we can build on \eqref{eq: d_1_est_begin} in the following manner:
\begin{flalign*}d_1^c(\{u_t\}_t,\{\textup{max}_\mathcal R(u_t,v_t)\}_t)&=\lim_{l \to \infty} \frac{I(\max(u_l,v_l))}{l} - I(u_1) = \lim_{l \to \infty} \frac{I(\max(u_l,v_l)) - I(u_l)}{l} \\
&= \lim_{l \to \infty} \frac{d_1(\max(u_l,v_l),u_l)}{l} \leq C\lim_{l \to \infty} \frac{d_1(u_l,v_l)}{l} = C d_1^c(\{u_t\}_t,\{v_t\}_t),
\end{flalign*}
where in the last inequality we have used Proposition \ref{remark 5.6: adaptation Dar15}. Using symmetry of $\{u_t\}_t,\{v_t\}_t$, the proof is finished.
\end{proof}

\section{The metric geometry of singularity types}\label{sect: metric geo}

The aim of this section is to show that $\mathcal S(X,\theta)$ embeds naturally in $\mathcal R(X,\theta)$, endowing the former space with a natural pseudo-metric structure.

Given $\psi \in \textup{PSH}(X,\theta)$ with $\psi \leq 0$, generalizing the methods of \cite[Section 4]{Dar13}, it is possible to define a geodesic ray $\{r[\psi]_t\}_t \in \mathcal R(X,\theta)$ whose potentials have minimal singularities. The specific construction is as follows.  Let $[0,l] \ni t \to r(\psi)^l_t \in \mathcal E^1(X,\omega)$ be the geodesic segment with minimal singularity type joining $r(\psi)^l_0=V_\theta$ and $r(\psi)^l_l = \max(\psi,V_\theta-l)$. Using the comparison principle (\cite[Proposition 3.2]{DDL1}) numerous times, it can be shown that for any fixed $t > 0$ the family $\{r(\psi)^l_t \}_{l \geq 0}$ is increasing as $l \to \infty$, and its limit equals the geodesic ray with minimal singularity type $t \to r[\psi]_t$. Along the way we also obtain the lower bound $\max(\psi,V_\theta-t) \leq r[\psi]_t$ for all $t \in [0,\infty)$.

Since $\psi \leq \psi'$ implies that $r[\psi]_t \leq r[\psi']_t$ and $r[\psi]_t = r[\psi+C]_t$, $C\in \mathbb{R}$,  we obtain that the construction of the ray only depends on the singularity type, giving us a map:
\begin{equation}\label{eq: S_to_R_map}
r[\cdot]:\mathcal S(X,\theta) \to \mathcal R(X,\theta).
\end{equation}
The basic idea will be to pull back the metric geometry of $\mathcal R(X,\theta)$ recalled in the previous section to $\mathcal S(X,\theta)$ via this map. Before we do this we carry out some preliminary analysis.
Since $[0,\infty) \ni t \to \max(\psi,V_\theta-t)$ is a subgeodesic ray with minimal singularities we have that $t \to I(\max(\psi, V_\theta-t))$ is convex by \cite[Theorem 3.8]{DDL1} and non-increasing by Theorem \ref{thm: basic I energy}. 

Via our embedding in \eqref{eq: S_to_R_map}, we can introduce the Monge-Amp\`ere energy of singularity types $$I_\mathcal S[\psi] := I\{r[\psi]_t\}.$$ 
\begin{theorem}  For $\psi \in \textup{PSH}(X,\theta)$ we have
\begin{equation}\label{eq: slope formula}
I_\mathcal S[\psi]= -\int_{X} \theta_{V_\theta}^n+  \frac{1}{n+1}\sum_{j=0}^n  \int_X \theta_{V_\theta}^j \wedge \theta_{\psi}^{n-j}. 
\end{equation}
\end{theorem}
\noindent The proof of the above theorem is analogous to  \cite[Theorem 2.5]{Dar13} that deals with the case when $\theta$ is K\"ahler. Nevertheless, given its central role in this work, we are going to give the details for the reader's convenience.
\begin{proof}

We can  assume w.l.o.g. that $\psi\leq V_\theta\leq  0$. Setting $\psi_t := \max(\psi,V_{\theta}-t)$, by \cite[Lemma 3.15]{DDL1} we have that
$$I_{\mathcal{S}}[\psi] =  \lim_{t\to +\infty} \frac{I(r[\psi]_t)}{t}= \lim_{t\to +\infty} \frac{I(\psi_t)}{t} =  \frac{1}{n+1}\sum_{k=0}^n \lim_{t\rightarrow +\infty}\int_X \frac{\psi_t-V_\theta}{t}\, \theta_{\psi_t}^k \wedge \theta_{V_\theta}^{n-k}.$$

Note that $\int_X \theta_{\psi_t}^n = \int_X \theta_{V_{\theta}}^n$ since $\psi_t$ has minimal singularity. By Lemma \ref{lem: plurifine},  we have 
\begin{flalign*}
\int_X \frac{\psi_t - V_\theta}{t} \theta_{\psi_t}&= \int_{\{\psi>V_\theta-t\}}\frac{\psi_t - V_\theta}{t} \theta_{\psi_t}^n -\int_{\{\psi\leq V_\theta-t\}} \theta_{\psi_t}^n\\
& = \int_{\{\psi>V_\theta-t\}}\frac{\psi - V_\theta}{t} \theta_{\psi}^n -\int_X \theta_{\psi_t}^n + \int_{\{\psi > V_\theta-t\}} \theta_{\psi_t}^n\\
& = \int_{\{\psi>V_\theta-t\}}\frac{\psi - V_\theta}{t} \theta_{\psi}^n -\int_X \theta_{V_{\theta}}^n + \int_{\{\psi > V_\theta-t\}} \theta_{\psi}^n. 
\end{flalign*}
The functions $\mathbbm{1}_{\{\psi>V_{\theta} -t\}} \frac{\psi - V_\theta}{t}$ are uniformly bounded in $[-1,0]$ and they converge pointwise to $0$ outside a pluripolar set on which the measure $\theta_{\psi}^n$ vanishes. Hence
$$
\lim_{t\rightarrow +\infty}\int_X \frac{\psi_t - V_\theta}{t} \,\theta_{\psi_t}^n=  - \int_X \theta_{V_{\theta}}^n + \int_X \theta_{\psi}^n.
$$
Since for any $j=1, \dots, n$, $\int_X \theta_{\psi_t}^{n-j} \wedge \theta_{V_{\theta}}^j= \int_X \theta_{V_{\theta}}^n$, the exact same arguments give
$$
\lim_{t\rightarrow +\infty}\int_X \frac{\psi_t - V_\theta}{t} \, \theta_{\psi_t}^{n-j}\wedge \theta_{V_\theta}^j= - \int_X \theta_{V_{\theta}}^n  + \int_X \theta_{\psi}^{n-j} \wedge \theta_{V_{\theta}}^j, \quad j=1, \dots, n . 
$$
This gives the conclusion.
\end{proof}

Finally, we list and prove the properties of the map $r[\cdot]$ that will be most important to us, finding a link with the ceiling operator $\mathcal C$ in the process:
\begin{prop}\label{prop: ray_ceiling_prop} Suppose $[\psi],[\chi] \in \mathcal S(X,\theta)$ such that $\psi,\chi \leq 0$. Then the following hold:\\
\noindent (i) $r[\psi]_\infty:=\lim_{t \to \infty} r[\psi]_t = \mathcal{C}(\psi)$.\\
\noindent (ii) $r[\psi]_t = r[\mathcal C(\psi)]_t$. In particular, $r[\mathcal C(\psi)]_\infty = \mathcal C[\psi]$.\\
\noindent (iii) $r[\psi]_t=r[\chi]_t$ if and only if $\mathcal C(\psi)=\mathcal C(\chi)$.\\
\noindent (iv) $P{[\mathcal C(\psi)]} = \mathcal C(\psi)$.
\end{prop}

In particular, part (i) of this proposition proves that the image of the ceiling operator $\mathcal C$ is exactly the collection of $\theta$-psh functions that can arise as time limits of geodesic rays of the type $\{r[\psi]_t\}_t$.

\begin{proof} 
First, we claim that given $u,v\in \textup{PSH}(X, \theta)$ such that $v \in \mathcal{F}_u$ then  $r[u]= r[v]$, and in particular $r[u]_\infty= r[v]_\infty$. Indeed, by \eqref{eq: slope formula} we know that $\lim_{t\rightarrow +\infty} \frac{I(r[u]_t)}{t}= \lim_{t\rightarrow +\infty} \frac{I(r[v]_t)}{t}$. The linearity of the energy $I$ \cite[Theorem 3.12]{DDL1} then insures that $I(r[v]_\ell)= I(r[u]_\ell)$ for any $\ell\geq 0$. Since $r[u]_\ell\leq r[v]_\ell $ it follows from Lemma \ref{lem: energy equal} that $r[u]_\ell= r[v]_\ell$, $\forall \ell \geq 0$.

Now, set $u:= r[\psi]_\infty$ and we claim that  $r[u]_t= r[\psi]_t$.
Since $u \geq \psi$, we get that $r[u]_t \geq r[\psi]_t$. For the other direction, we note that $r[\psi]_t\geq u$ and $r[\psi]_t\geq V_\theta-t$, thus $r[\psi]_t\geq \max(u, V_\theta-t)$. The inequality $r[u]_t \leq  r[\psi]_t$ then follows by the construction of $\{r[u]_t\}_t$ together with the comparison principle. 

As a consequence of this second claim we have that $I(r[u]_t)= I(r[\psi]_t)$. From \eqref{eq: slope formula} and \cite[Proposition 3.1]{DDL2} we get
$$\int_X \theta_u^k\wedge \theta_{V_\theta}^{n-k}= \int_X \theta_\psi^k\wedge \theta_{V_\theta}^{n-k}, \quad \forall k=0,\cdots, n.$$
This means that $u\leq \mathcal{C}(\psi)$. Moreover, our claims give
$r[\psi]_t=r[u]_t=r[\mathcal{C}(\psi)]_t$ for any $t$, hence $u=r[u]_\infty =r[\mathcal{C}(\psi)]_\infty \geq \mathcal{C}(\psi)$. Thus $u= \mathcal{C}(\psi)$ addressing $(i)$. The statement in $(ii)$ is just a consequence of the second claim above.

 If $r[\psi]=r[\chi]$, then by $(i)$ we have that $\mathcal C(\psi)=r[\psi]_\infty=r[\chi]_\infty=\mathcal C(\chi)$. Conversely if $\mathcal C(\psi)=\mathcal C(\chi)$, then by $(ii)$ we have $r[\psi]=r[\chi].$ This proves $(iii)$.

Lastly, since $\inf_{t \in [0,\infty)} r[\psi]_t =r[\psi]_\infty=\mathcal C(\psi)$, by \cite[ Lemma 3.17]{DDL1} we have that $P{[\mathcal C(\psi)]}=\mathcal C(\psi)$, establishing $(iv)$.
\end{proof}

Finally, as previously advertised, we consider the $L^1$ (pseudo)metric structure of $\mathcal S(X,\theta)$, by pulling back the chordal metric structure from $\mathcal R(X,\theta)$:
$$d_\mathcal S([\psi],[\chi]):= d^c_1(\{r[\psi]_t\}_t,\{r[\chi]_t\}_t).$$
Our main result about $(\mathcal S(X,\theta),d_\mathcal S)$ in this subsection characterizes the singularity types that are at zero distance apart with respect to the $d_{\mathcal S}$ pseudo-metric:
\begin{theorem}\label{criteria d_1 vanishing}
$(\mathcal S(X,\theta),d_\mathcal S)$ is a pseudo-metric space. More precisely, the following are equivalent:\\
\noindent (i) $d_\mathcal S([\psi],[\chi])=0$. \\
\noindent (ii) $r [\psi]=r[\chi]$.\\
\noindent (iii) $\mathcal C(\psi)=\mathcal C(\chi)$. 
\end{theorem}
\begin{proof} Theorem \ref{thm: R_complete}  gives the equivalence $(i) \Leftrightarrow (ii)$. Proposition \ref{prop: ray_ceiling_prop}$(iii)$ gives that $(ii) \Leftrightarrow (iii)$. 

\end{proof}

In case $[u] \leq [v]$, using \eqref{eq: slope formula}, the expression for $d_\mathcal S([u],[v])$ is especially simple : 
\begin{lemma}\label{lem: d_s_monotone} If $[u], [v] \in \mathcal S(X,\theta)$ is such that $[u] \leq [v]$ then
$$d_\mathcal S([u],[v]) = \frac{1}{n+1}  \sum_{j=0}^n \bigg(\int_X \theta^j \wedge \theta_v^{n-j}-\int_X \theta^j \wedge \theta_u^{n-j} \bigg).
$$
\end{lemma}
Observe that, it is a direct consequence of the above that if $[u]\leq [v]\leq [w]$ then $d_\mathcal S([u],[w])\geq d_\mathcal S([v],[w]) $.

When $[u] \not \leq [v]$, then a similar simple expression for $d_\mathcal S([u],[v])$ may not be available, however one can find a useful expression that totally governs the behavior of $d_\mathcal S([u],[v])$, as shown in Proposition \ref{prop: poor_Pythagorean_S}  below.

It is clear that for $[u],[v] \in \mathcal S(X,\theta)$ it makes sense to define $[\max(u,v)] \in \mathcal S(X,\theta)$, which does not depend on the choice of representatives $u,v \in \textup{PSH}(X,\theta)$. We start by arguing that 
\begin{equation}\label{eq: max_R_S_are_the_same}
\textup{max}_\mathcal R(r[u]_t,r[v]_t) = r[\max(u,v)]_t, \ \  t \geq 0,
\end{equation}
where $\{\textup{max}_\mathcal R(r[u]_t,r[v]_t)\}_t$ is the smallest ray that lies above $\max(r[u]_t,r[v]_t)$ that was constructed in Section \ref{sec: chordal geometry}. But $\{r[\max(u,v)]_t\}_t$ has this ``extremal" property as well. Indeed by construction we have $\max(r[u]_t,r[v]_t) \leq r[\max(u,v)]_t$, and any ray $\{w_t\}_t$ that satisfies $\max(r[u]_t,r[v]_t) \leq w_t$ has to also satisfy $r[\max(u,v)]_t \leq w_t$. 

Next we notice that \eqref{eq: max_R_S_are_the_same} and Proposition \ref{prop: poor_Pythagorean_R} allow to establish the following decomposition inequality for the $d_\mathcal S$ pseudo-metric:
\begin{prop} \label{prop: poor_Pythagorean_S} There exists $C>1$ such that, for all $[u],[v] \in \mathcal S(X,\theta)$,
\begin{equation}
 d_\mathcal S([u],[v]) \leq d_\mathcal S([u],[\max(u,v)]) + d_\mathcal S([\max(u,v)],[v])\leq C d_\mathcal S([u],[v]).
\end{equation}
\end{prop}

As $d_\mathcal S([\psi],[\phi]) = I_\mathcal S[\psi] - I_\mathcal S[\phi]$ if $[\psi] \geq [\phi]$, we note the following corollary of \eqref{eq: slope formula}, somewhat reminiscent of \cite[Proposition 4.9]{Dar15}:
\begin{lemma}\label{lem: monotone_lim_S} Suppose $u_j,u \in \textup{PSH}(X,\theta)$ are such that  either $[u_j] \leq [u]$ or $[u] \leq [u_j]$. Then  $d_\mathcal S([u_j],[u]) \to 0$ if and only if $\int_X \theta_{V_\theta}^k \wedge \theta_{u_j}^{n-k} \to \int_X \theta_{V_\theta}^k \wedge \theta_{u}^{n-k}, \ k \in \{0,\ldots,n\}$.
\end{lemma}

It follows from \eqref{eq: I_R_cont} that $|I_\mathcal S[u] - I_\mathcal S[v]| \leq d_\mathcal S([u],[v])$. By the next lemma it turns out that a similar statement holds for the individual components of the sum in \eqref{eq: slope formula} as well:
\begin{lemma}\label{lem: mixed_MA_dC_conv} There exists $C>1$ such that for all $k \in \{1,\ldots,n \}$ and $[u],[v] \in \mathcal S(X,\theta)$ we have
$$\bigg|\int_X \theta_{V_\theta}^k \wedge \theta^{n-k}_u -  \int_X \theta_{V_\theta}^k \wedge \theta^{n-k}_v\bigg| \leq C d_\mathcal S([u],[v]).$$
\end{lemma}
\begin{proof} If $[\psi] \leq [\phi]$ then \cite[Theorem 1.1]{DDL2} implies that $\int_X \theta_{V_\theta}^k\wedge \theta^{n-k}_\psi \leq  \int_X \theta_{V_\theta}^k\wedge \theta^{n-k}_\phi, \ k \in \{0,\ldots,n \}$. As a result of this and \eqref{eq: slope formula} we can write that
\begin{equation*}
\frac{1}{n+1}\bigg|\int_X \theta_{V_\theta}^k\wedge \theta^{n-k}_\psi -  \int_X \theta_{V_\theta}^k\wedge \theta^{n-k}_\phi\bigg| \leq \big|I_\mathcal S[\psi] - I_\mathcal S[\phi]\big | = d_\mathcal S([\psi],[\phi]).
\end{equation*}
For general $[u],[v] \in \mathcal S(X,\theta)$, using this last inequality and Proposition \ref{prop: poor_Pythagorean_S} we can conclude:
\begin{flalign*}
\bigg|\int_X  \theta_{V_\theta}^k \wedge \theta^{n-k}_u &-  \int_X \theta_{V_\theta}^k\wedge \theta^{n-k}_v\bigg|  \\
&\leq \bigg|\int_X \theta_{V_\theta}^k\wedge \theta^{n-k}_u -  \int_X \theta_{V_\theta}^k\wedge \theta^{n-k}_{\max(u,v)}\bigg| + \bigg|\int_X \theta_{V_\theta}^k\wedge \theta^{n-k}_{\max(u,v)} -  \int_X \theta_{V_\theta}^k\wedge \theta^{n-k}_v\bigg|\\
& \leq (n+1)d_\mathcal S([u],[\max(u,v)]) + (n+1)d_\mathcal S([\max(u,v)],[v])\\
& \leq (n+1)C d_\mathcal S([u],[v]). 
\end{flalign*}
\end{proof}

\section{Discussion of completeness of singularity types}\label{sect: completeness}
In this section we prove the completeness  of the spaces $(\mathcal S_\delta(X,\theta),d_\mathcal S)$, for $\delta>0$.
Recall from the introduction that $\mathcal S_\delta(X,\theta)=\{[u] \in \mathcal S(X,\theta) \textup{ s.t. } \int_X \theta_u^n \geq \delta \}$. On the other hand, we will also show that $(\mathcal S(X,\theta),d_\mathcal S)$ is not complete.

First we show that ``increasing" sequences always have a $d_\mathcal S$-limit inside $\mathcal S(X,\theta)$:

\begin{lemma}\label{lem: mon_limit_complete} Let  $u_j \in \textup{PSH}(X,\theta)$ such that $\{u_j \}_j$ is increasing  and $u_j \leq 0$. Then $d_\mathcal S([u_j],[u]) \to 0$, where $u_j \nearrow u$ a.e. on $X$.
\end{lemma}
\begin{proof}From \cite[Theorem 1.2, Remark 2.5]{DDL2}, it follows that $\int_X \theta_{V_\theta}^k \wedge \theta^{n-k}_{u_j} \to \int_X \theta_{V_\theta}^k \wedge \theta^{n-k}_{u}$. Since $[u_j] \leq [u]$, the proof is finished after an application of Lemma \ref{lem: monotone_lim_S}.
\end{proof}

\begin{prop}\label{prop: Cauchy_monotone}
Suppose that $[u_j] \in  \mathcal S(X,\theta)$ is a $d_\mathcal S$-Cauchy sequence with $u_j \leq 0$. Then there exists a decreasing sequence $\{[v_j]\}_j \subset \mathcal S(X,\theta)$ that is equivalent with $\{[u_j]\}_j$, i.e. $d_\mathcal S([u_j],[v_j])\to 0$ as $j \to \infty$. 
\end{prop}

\begin{proof}We can assume that $d_\mathcal S([u_j],[u_{j+1}]) \leq C^{-2j}$, where $C$ is the constant of Proposition \ref{prop: poor_Pythagorean_S}. Let 
\begin{equation}\label{eq: v_j_def}
v_j := \textup{usc}\Big(\sup_{k \geq j} u_k\Big) \in \textup{PSH}(X,\theta).
\end{equation}
Lemma \ref{lem: mon_limit_complete} implies that $\lim_l d_\mathcal S([v^l_j],[v_j]) = 0$, where 
$v^l_j =\sup_{k \in \{j,\ldots, j+l\}} u_k$.
To finish the argument, we show that $ \lim_{l}d_\mathcal S([u_j],[v_j^l])=d_\mathcal S([u_j],[v_j])\to 0$ as $j \to \infty$. Using the triangle inequality and Proposition \ref{prop: poor_Pythagorean_S} we get
\begin{flalign*}
d_\mathcal S([u_j],[v_j^l]) &= d_\mathcal S([u_j],[\max(u_j,v^{l-1}_{j+1})]) \leq C d_\mathcal S([u_j],[v^{l-1}_{j+1}])\\ 
& \leq C(d_\mathcal S([u_j],[u_{j+1}]) + d_\mathcal S([u_{j+1}],[v^{l-1}_{j+1})]).
\end{flalign*}
After iterating the above inequality $l$ times and observing that $d_\mathcal S([u_{j+l}], [v_{j+l}^0])=0$, we conclude that  
\begin{flalign}\label{eq: v_j_u_j_C}
d_\mathcal S([u_j],[v_j^l]) &\leq \sum_{k \in \{j,\ldots,j+l-1\}} C^{k+1-j} d_\mathcal S([u_k],[u_{k+1}]) = \sum_{k \in \{j,\ldots,j+l-1\}} C^{k+1-j} \frac{1}{C^{2k}} \nonumber\\
&\leq \sum_{k \geq j} C^{k+1-j} \frac{1}{C^{2k}}=\sum_{k \geq j} \frac{1}{C^{k+j -1}}  \leq \frac{1}{C^{j-1}}\frac{C}{C-1}.
\end{flalign}
\end{proof}

\subsection{Completeness of $\mathcal S_\delta(X,\theta)$}

Under the assumption of non-collapsing mass, we will show below that decreasing $d_\mathcal S$-Cauchy sequences do indeed converge, implying completeness of $\mathcal S_\delta(X,\theta)$, via Proposition \ref{prop: Cauchy_monotone}. For this we will need the following important lemma:

\begin{lemma}\label{lem: subextension}
Assume that $u,v\in \psh(X,\theta)$, $u \leq v$, $\int_X \theta_u^n>0$ and $b>1$ is such that 
\begin{equation}\label{eq: non-collapsing n}
b^n\int_X \theta_{u}^n>(b^n-1)\int_X \theta_{v}^n.
\end{equation}
Then $P(bu+(1-b)v) \in \psh(X,\theta)$. 
\end{lemma}

\begin{proof}
If $P(bu+(1-b)\mathcal{C}(v)) \in \PSH(X,\theta)$ then so does $P(bu+(1-b)v)$ since $v \leq \mathcal{C}(v)$ and $(1-b)<0$. Therefore,  after possible replacing $v$ with $\mathcal C(v)$, we can assume that $v=\mathcal{C}(v)$. Since $\int_X \theta_v^n \geq \int_X \theta_u^n > 0$, we have that $P[v]=\mathcal C(v)$ (\cite[Remark 2.5, Theorem 3.12]{DDL2}).

For $j\in \mathbb{N}$ we set $u_j:=\max(u,v-j)$ and $\varphi_j:=P(bu_j+(1-b)v)$. Observe that $\varphi_j$ is a decreasing sequence of $\theta$-psh functions, whose singularity type is equal to $[v]$. The proof is finished if we can show that $\lim_j \varphi_j \not \equiv -\infty$ is a $\theta$-psh function. Assume by contradiction that $\sup_X \varphi_j \to -\infty$.  
It follows from Lemma \ref{lem: basic MA inequality} below that 
\begin{equation} \label{eq: sub extension 1}
\theta_{\varphi_j}^n \leq b^n\mathbbm{1}_{\{\varphi_j = bu_j+(1-b)v\}}\theta_{u_j}^n. 
\end{equation}
Fix $j>k>0$.  We note that $u_j=u$ on $\{u>v-k\}$ and, since  $u_j$ has singularity type equal to $[v]$, we have by \cite[Theorem 1.2]{WN19} and the plurifine locality, 
$$
\int_{\{u\leq v -k\}} \theta_{u_j}^n = \int_X \theta_{u_j}^n -\int_{\{u>v-k\}} \theta_{u_j}^n=\int_X \theta_{v}^n -\int_{\{u>v-k\}} \theta_{u}^n.
$$
Since $\{u_j \leq v-k\} = \{u\leq v-k\}$, from the above and \eqref{eq: sub extension 1} we obtain  
\begin{eqnarray}\label{eq: sub extension 2}
\theta_{\varphi_j}^n(\varphi_j\leq v-bk) &\leq & b^n \mathbbm{1}_{\{\varphi_j= bu_j+(1-b)v\}} \theta_{u_j}^n( \varphi_j\leq v-bk) \leq  b^n \theta_{u_j}^n (bu_j+(1-b)v \leq v-bk) \nonumber \\
& = & b^n \theta_{u_j}^n (u_j\leq v-k)\leq   b^n \theta_{u_j}^n (u\leq v-k) \leq  b^n\Big(\int_X \theta_{v}^n -\int_{\{u>v-k\}} \theta_u^n\Big).
\end{eqnarray}
Since $v = P[v]$ we have $\sup_X ( \varphi_j-v )=\sup_X \varphi_j\to -\infty$ by Lemma \ref{lem: compactness and model type}.  From this we see that $\{\varphi_j\leq v-bk\}=X$ for $j$ large enough, $k$ being fixed. Thus, letting $j\to +\infty$ in \eqref{eq: sub extension 2}, and then $k\to +\infty$ gives
$$
\int_X \theta_v^n \leq b^n \left (\int_X \theta_{v}^n -\int_X \theta_u^n\right),
$$
contradicting  \eqref{eq: non-collapsing n}.  Consequently,  $\varphi_j$ decreases to a  $\theta$-psh function, finishing the proof.
\end{proof}

\begin{lemma}\label{lem: basic MA inequality}
Assume that $b \geq 1$, and $u,v,P_{\theta}(bu+(1-b)v)\in \psh(X,\theta)$. Then the measure $\theta_{P_{\theta}(bu+(1-b)v)}^n$ is supported on the contact set $\{P_{\theta}(bu+(1-b)v)=bu+(1-b)v \}$, and
\begin{equation}\label{eq: measure_ineq_b}
\theta_{P_{\theta}(bu+(1-b)v)}^n \leq b^n\theta_{u}^n. 
\end{equation}
\end{lemma}
\begin{proof}
Up to rescaling, we can assume that $\theta \leq \omega$,  and hence $\PSH(X,\theta) \subset \PSH(X,\omega)$.  Let $u_j \in C^\infty(X)\cap \PSH(X, \omega)$ be such that $u_j \searrow u$. This is possible thanks to \cite{BK07}, \cite{Dem92}.  Set $\psi_j:=P_{\theta}(bu_j+(1-b)v)$ and $\psi:=P_{\theta}(bu+(1-b)v)$ and note that $\psi_j\searrow  \psi$. Also, $bu_j+(1-b)v$ is lower semicontinuous hence the set $\{\psi_j<bu_j+(1-b)v\}$ is open. By a classical balayage argument, for each $j\in \NN$ the measure $\theta_{\psi_j}^n$ vanishes on the set  $\{\psi_j<bu_j+(1-b)v\}$. Since $\psi_j\leq bu_j+(1-b)v$ we have
$$
\int_X \min (bu_j+(1-b) v- \psi_j, 1) \theta_{\psi_j}^n =0.
$$
The functions in the integral are uniformly bounded with values in $[0,1]$, quasi-continuous, and (since $u_j$ and $\psi_j$ are $\omega$-psh functions decreasing  to $u$ and $\psi$ respectively) they converge in capacity to  $\min(bu+(1-b)v -\psi, 1)$, which is quasi-continuous and bounded on $X$. It follows from Theorem \ref{thm: lsc of non pluripolar product} that after letting $j\to +\infty$ in the above equality we obtain 
$$
\int_X \min (bu+(1-b) v- \psi, 1) \theta_{\psi}^n\leq \liminf_{j} \int_X \min (bu_j+(1-b) v- \psi_j, 1) \theta_{\psi_j}^n=0. 
$$
This implies that $\theta_{\psi}^n$ vanishes in the set $\{\psi<bu+(1-b)v\}$.

Now we prove \eqref{eq: measure_ineq_b}. 
Let $\varphi:= \frac{1}{b} P_{\theta}(bu+(1-b)v) + \left(1- \frac{1}{b}\right) v$. Note that $\varphi\leq u$. By Lemma \ref{lem: concentration max} below we then have
$$\mathbbm{1}_{\{\varphi=u \}} \theta_{\varphi}^n \leq \mathbbm{1}_{\{\varphi=u\}} \theta_{u}^n \leq \theta_{u}^n.$$
Moreover, by the above we know that $\theta_{P_{\theta}(bu+(1-b)v)}^n$ is supported on  $\{\varphi=u \}$, 
hence 
$$\frac{1}{b^{n}}\theta_{P_{\theta}(bu+(1-b)v)}^n= \frac{1}{b^{n}}\mathbbm{1}_{\{\varphi=u \}}\theta_{P_{\theta}(bu+(1-b)v)}^n\leq \mathbbm{1}_{\{\varphi=u \}}\theta_{\varphi}^n \leq \theta_{u}^n.$$
\end{proof}

\begin{lemma}\label{lem: concentration max} Let $\varphi,\psi\in \psh(X,\theta)$. Then
\begin{equation}\label{eq: Demailly_est}
\theta_{\max(\varphi, \psi)}^n \geq \mathbbm{1}_{\{\psi \leq \varphi\}}\theta_\varphi^n + \mathbbm{1}_{\{\varphi <\psi\}}\theta_\psi^n.
\end{equation}
In particular, if $\varphi \leq \psi$ then $\mathbbm{1}_{\{\varphi=\psi\}} \theta_\varphi^n \leq \mathbbm{1}_{\{\varphi=\psi\}} \theta_\psi^n.$
\end{lemma}
\begin{proof} Let $\psi_k := \max(\psi, V_\theta -k)$ and $\varphi_k := \max(\varphi, V_\theta -k)$.
It follows from \cite[Theorem 2.2.10]{bl} that 
$$\theta_{\max(\varphi_k, \psi_k)}^n \geq \mathbbm{1}_{\{\psi_k \leq \varphi_k\}}\theta_{\varphi_k}^n + \mathbbm{1}_{\{\varphi_k <\psi_k\}}\theta_{\psi_k}^n.$$
Multiplying with $\mathbbm{1}_{\{\varphi > V_\theta -k \} \cap  \{ \psi > V_\theta -k \}}$, and using plurifine locality (Lemma \ref{lem: plurifine})  we arrive at
$$\mathbbm{1}_{\{\varphi > V_\theta -k \} \cap  \{ \psi > V_\theta -k \}}\theta_{\max(\varphi, \psi)}^n \geq \mathbbm{1}_{\{\varphi > V_\theta -k \} \cap  \{ \psi > V_\theta -k \}\cap \{\psi \leq \varphi\}}\theta_{\varphi}^n + \mathbbm{1}_{\{\varphi > V_\theta -k \} \cap  \{ \psi > V_\theta -k \}\cap \{\varphi <\psi\}}\theta_{\psi}^n.$$
Letting $k \to \infty$, \eqref{eq: Demailly_est} follows. 
\end{proof}

Next we prove that along a decreasing sequence of fixed points of $\mathcal C$  the total masses converge.

\begin{prop}\label{prop: MA dec maxi} Let $u_j,u \in \textup{PSH}(X,\theta)$ be such that $\sup_X u_j=\sup_X u=0$, and $u_j$ converges in capacity to  $u$. Then
$$
\limsup_{j\rightarrow +\infty} \int_{\{u_j =0\}} \theta_{u_j}^n \leq \int_{\{u=0\}} \theta_u^n.
$$
If additionally $u_j=\mathcal{C}(u_j)$ then  
$
\lim_{j\to +\infty}\int_X \theta_{u_j}^n = \lim_{j\to +\infty}\int_{\{u_j =0\}} \theta_{u_j}^n =\int_{\{u=0\}} \theta_{u}^n  =\int_X \theta_{u}^n.
$
\end{prop}

\begin{proof}
For each $C>0$ and each $v\in \psh(X,\theta)$ we set $v^C:=\max(v,V_\theta-C)$. Since $u_j \leq V_{\theta}$, we have that $\{u_j=0\} \subset \{V_{\theta} =0\}$. 
For each $\beta>0$, using plurifine locality (Lemma \ref{lem: plurifine}) we can write
\begin{eqnarray*}
\limsup_{j\to +\infty} \int_{\{u_j=0\}} \theta_{u_j}^n =  \limsup_{j\to +\infty}\int_{\{u_j=0\}} \theta_{u_j^C}^n  \leq  \limsup_{j\to +\infty}\int_{\{V_{\theta}=0\}} e^{\beta u_j}\theta_{u_j^C}^n\leq \int_{\{V_{\theta}=0\}} e^{\beta u}  \theta_{u^C}^n,\end{eqnarray*}
where the last inequality follows from \cite[Theorem 4.26]{GZ17}. Indeed,   since $e^{\beta u_j}$ is a sequence of bounded quasi-continuous functions converging in capacity to $e^{\beta u}$ and $u_j^C$ converges in capacity to $u^C$ all having minimal singularities $V_{\theta} -C \leq u_j^C\leq 0$, by \cite[Theorem 4.26]{GZ17} we have that $e^{\beta u_j} \theta_{u_j^C}^n$ converges weakly to $e^{\beta u} \theta_{u^C}^n$ in $\Omega$, the ample locus of $\{\theta\}$. Since $\{V_{\theta}=0\}$ is a compact subset of $\Omega$, the last inequality follows. 

Letting $\beta \to +\infty$ and noting that $e^{\beta u}$ decreases to $\mathbbm{1}_{\{u=0\}}$ we arrive at 
$$
\limsup_{j\to +\infty} \int_{\{u_j=0\}} \theta_{u_j}^n \leq \int_{\{u=0\}} \theta_{u^C}^n =\int_{\{u=0\}} \theta_{u}^n,
$$
where we used again Lemma \ref{lem: plurifine}. This finishes the proof of the first part. 

Assume now that $u_j = \mathcal{C}(u_j)$. In case $\int_X \theta_{u_j}^n > 0$ then we have $\mathcal C(u_j)=P[u_j]$, and by \cite[Theorem 3.8]{DDL2} we know that $\theta_{u_j}^n$ is supported on the contact set $\{u_j=0\}$. In case $\int_X \theta_{u_j}^n = 0$ this same fact is trivially true.

Observe that the inequality $\liminf_j \int_X \theta_{u_j}^n \geq\int_X \theta_{u}^n$ follows from \cite[Theorem 2.3]{DDL2}. From this and the previous step we obtain 

\begin{equation} \label{eq: MA dec maxi}
\int_X \theta_{u}^n \leq \liminf_j \int_X \theta_{u_j}^n \leq \limsup_j \int_X \theta_{u_j}^n = \limsup_j \int_{\{u_j=0\}} \theta_{u_j}^n \leq  \int_{\{u=0\}} \theta_u^n \leq \int_X \theta_{u}^n.
\end{equation}
The conclusion follows.
\end{proof}

\begin{coro}\label{coro: C stable increasing}
Assume that $u_j \in \PSH(X,\theta)$ and $u_j =\mathcal{C}(u_j)$. If  $u_j \searrow u$, then $u= \mathcal{C}(u)$. If $u_j$ converges in capacity to $u$ and $\int_X \theta_u^n>0$ then $u= \mathcal{C}(u)$. 
\end{coro}

\begin{proof}
If $u_j \searrow u$ then $u \leq \mathcal C(u)\leq \mathcal C(u_j)=u_j$, hence $u= \mathcal{C}(u)$. Assume now that  $u_j$ converges in capacity to $u$ and $\int_X \theta_u^n>0$. It follows from \eqref{eq: MA dec maxi} that $\int_{\{u=0\}} \theta_u^n =\int_X \theta_u^n$. Hence $\theta_u^n$ is supported on $\{u=0\}=\{u=P[u]=0\}$, where we used that $u \leq P[u] \leq 0$. In particular, $u \geq P[u]$ a.e. with respect to $\theta_u^n$.
Since $\int_X \theta_{P[u]}^n=\int_X \theta_u^n >0$ we can use the domination principle, \cite[Proposition 3.11]{DDL2}, and Proposition \ref{prop: C_operator_properties}(i) to conclude that $u=P[u]=\mathcal{C}(u)$. 
\end{proof}

In the presence of non-vanishing mass,  we can also show the convergence of the mixed masses of decreasing model potentials:
\begin{prop}\label{prop: mixed MA dec maxi} Let $\delta>0$, \ $u_j,u \in \textup{PSH}(X,\theta)$ such that $\mathcal C(u_j)=u_j$ and $u_j\searrow u$. If $\int_X \theta_{u_j}^n\geq \delta$ then
\[
\lim_{j\to +\infty}\int_X \theta_{u_j}^m \wedge \theta_{V_\theta}^{n-m} = \int_X \theta_{u}^m\wedge  \theta_{V_\theta}^{n-m}, \  \ m\in \{0,...,n\}. 
\]
\end{prop}
\begin{proof} Let $\{b_j\}_j$, $b_j>1$, be an increasing sequence with $b_j \to +\infty$ such that
$b^n_j \int_X \theta_{u}^n \geq (b^n_j - 1)\int_X 
\theta_{u_j}^n.$ 
Such a sequence exists due to the fact that $\int_X \theta_{u_j}^n \to \int_X \theta_{u}^n \geq \delta$ (Proposition \ref{prop: MA dec maxi}). Lemma \ref{lem: subextension} gives that $v_j: = P(b_ju+(1-b_j)u_j) \in \textup{PSH}(X,\theta)$. Since $\frac{1}{b_j} v_j + \big(1 - \frac{1}{b_j}\big) u_j \leq u$ using \cite[Theorem 1.1]{DDL2} we obtain that 
$$\int_X \theta_{u_j}^m \wedge \theta_{V_\theta}^{n-m} \geq \int_X \theta_u^m \wedge \theta_{V_\theta}^{n-m} \geq \bigg(1 - \frac{1}{b_j}\bigg)^{m}\int_X \theta_{u_j}^m \wedge \theta_{V_\theta}^{n-m}.$$
Letting $j \to +\infty$ the result follows.
\end{proof}

Finally, we summarize the above findings in our main theorem:
\begin{theorem} Fix $\delta>0$. The pseudo metric space  $(\mathcal S_\delta(X,\theta),d_\mathcal S)$ is complete.
\end{theorem}
\begin{proof}Let $[u_j] \in \mathcal S_\delta(X,\theta)$ be a $d_\mathcal S$-Cauchy sequence. By \cite[Theorem 1.1]{DDL2}, the decreasing $d_\mathcal S$-Cauchy sequence $[v_j]$ that is equivalent to $[u_j]$   (constructed in Proposition \ref{prop: Cauchy_monotone}) also satisfies $v_j\geq u_j$ and $\int_X \theta^n_{v_j} \geq \delta$, hence $[v_j] \in \mathcal S_\delta(X,\theta)$. Since $\dS([v_j],[\mathcal{C}(v_j)]) =0$, we can assume that $v_j=\mathcal C(v_j)$. Set $v = \lim_j \mathcal C(v_j)$. Proposition \ref{prop: mixed MA dec maxi} together with Lemma \ref{lem: monotone_lim_S} imply that $d_\mathcal S([v_j],[v]) \to 0$. By Proposition \ref{prop: MA dec maxi}, $[v] \in \mathcal S_\delta(X,\theta)$, finishing the proof.
\end{proof}

\subsection{Incompleteness of $\mathcal S(X,\theta)$} \label{sect: incompleteness}

In this short section we show that $(\mathcal S(X,\omega),d_\mathcal S)$ is in general not complete. In particular, completeness fails even in the K\"ahler case. We start with the following general lemma:

\begin{lemma}\label{lem: P(tu)_model} Let $h \in \textup{PSH}(X,\omega)$ with $h = \mathcal C(h)$. If for $t > 1$ we have that $P_\omega(th) \in \textup{PSH}(X,\omega)$ then $P_\omega(th)=\mathcal C(P_\omega(th))$.
\end{lemma}

\begin{proof} Fix  $1 <s < t$ and assume that $P(th) \in \PSH(X,\omega)$. Let $v\leq 0$ be a $\omega$-psh function more singular than $P(sh)$. Then $s^{-1} v$ is more singular than $h$. Since $s^{-1} v$ is $\omega$-psh and $h =\mathcal C(h)= P [h]$, it follows that $s^{-1}v \leq h$, hence $v \leq sh$. Now, since $v$ is $\omega$-psh it follows that $v \leq P(sh)$. Hence $P[P(sh)]= P(sh)$. Since $\frac{s}{t}P(th) \leq P(sh)$ and  $s \in (1,t)$, the mass of $P(sh)$ is positive. 
It follows that $P(sh)$ is a model potential (Proposition \ref{prop: C_operator_properties}). 

Lastly, $P(t h)$ is the decreasing limit of the model potentials $P(su),  s \in (1,t)$. 
It then follows from Corollary \ref{coro: C stable increasing} that $P(th)$ is also a model potential, i.e. $\mathcal C(P(t h)) = P(t h)$.
\end{proof}

For the rest of this subsection assume that $X$ is equal to the complex surface described in \cite[Example 1.7]{DPS94}: $X := \Bbb P(E)$ where $E$ is a rank $2$  vector bundle over an elliptic curve $\Gamma = \Bbb C / \Bbb Z + \tau \Bbb Z, \textup{Im } \tau>0$. By \cite{DPS94} the line bundle $N:=\mathcal O_E(1)$ is nef and the only positive current of $N$ is a current of integration along a curve $C$, i.e. $ c_1(N)=\{[C]\}$.

Naturally we have a projection map $\pi: \Bbb P(E) \to \Gamma$. Since elliptic curves are projective, there exists an ample line bundle $M\to \Gamma$. The bundle $L:=\pi^* M^k \otimes \mathcal O_E(1)$ is ample over $X=\Bbb P(E)$ for big enough $k>0$ \cite[Proposition 7.10, page 161]{Har77}.

Let $\eta>0$ be the curvature form of $L$ and $\theta$  be the curvature of $N$. By possibly replacing $L$ with its high powers we can assume that $\eta + \theta$ is a K\"ahler form. Then  $\omega := 2 \eta + \theta$ is a K\"ahler form as well on $X$. Due to existence of sections for high powers of $L$, there exists $u \in \PSH(X,\eta)$, $\sup_X u=0$, such that $\eta_u:=\eta +i \ddbar u = [D]$ for some smooth curve $D \subset X$. In particular $\int_X \eta_u\wedge \alpha=0$, for any K\"ahler form $\alpha$.

Set  $\phi := \mathcal C_\omega(u) \in \PSH(X,\omega)\leq 0$ and note that 
$$
(\omega + i \ddbar u)^2 =(\eta+\theta+\eta_u)^2= (\eta + \theta)^2.
$$
It then follows from \cite[Theorem 3.3]{DDL4} that $u-\phi$ is bounded.

We now set $\phi_t := P_{\omega}(t\phi)$, for $t\in [1,2]$. Clearly, $2u + V_{\theta}$ is $\omega$-psh and it is smaller than $2\phi$, hence $\phi_t \in \PSH(X,\omega)$  and it is  a model potential thanks to Lemma \ref{lem: P(tu)_model}.

We estimate the mixed mass of $\phi_t$ for each $t\in [1,2)$.   Let $\varphi_t$ be a smooth negative $((2-t)\eta + \theta)$-psh function. Then $\phi_t$ is less singular than $tu+ \varphi_t \in \PSH(X,\omega)$,  hence 
\begin{eqnarray*}
\int_X (\omega +i\ddbar \phi_t ) \wedge \omega &\geq  &\int_X (2\eta + \theta +i\ddbar  (tu+\varphi_t)  ) \wedge \omega\\
&\geq  & \int_X ((2-t)\eta + \theta ) \wedge  (\eta  + (\eta + \theta)) 
 \geq   \{\theta\}. \{\eta\}=\int_C \eta>0,
\end{eqnarray*}
where in the third inequality above we used the fact that $\{\eta +\theta\}$ is a K\"ahler class. 

Also, $\phi_2 = P_\omega(2\phi)$ is more singular than $2u$. 
The potential $P_\omega(2\phi) -2u$ is then bounded from above and it satisfies:
$$
\theta + i \ddbar (P_\omega(2\phi) -2u) = \theta +2\eta + i \ddbar P_\omega(2\phi) \geq 0 \textup{ on } X \setminus D,
$$
since $\eta + i\ddbar u =0$ in $X\setminus D$. Therefore $\phi_2-2u$ extends over $X$ as a $\theta$-psh function. Thus $\phi_2 -2u = V_{\theta} + C_1$, for some constant $C_1$. 
Since $\theta +i\partial \overline{\partial}  V_{\theta}=[C]$ is a current of integration, the following holds for the mixed mass of $\phi_2$:
$$
\int_X (\omega+i \ddbar  \phi_2) \wedge \omega= \int_X (2\eta_u+ \theta_{V_\theta})\wedge (2\eta+\theta)=0.   
$$
Hence $\phi_s \searrow \phi_2$ but  due to Lemma \ref{lem: monotone_lim_S} we have that $d_{\mathcal{S}}([\phi_s],[\phi_2]) \not \rightarrow 0$ as $s\to 2$.

Due to Lemma \ref{lem: d_s_monotone} and the fact that $\phi_s$ has positive mass for any $s\in[1,2)$, we have that $\{[\phi_s]\}_{s \in [1,2)}$ does form a $d_\mathcal S$-Cauchy sequence. By Theorem \ref{thm: R_complete} there exists $\{r_t\}_t \in \mathcal R(X,\theta)$ such that $d_1^c(\{r[\phi_s]_t\}_t,\{r_t\}_t) \to 0$. The construction of the geodesic ray $\{r_t\}$ is explicit: for each $t>0$, $r_t$ is the limit in $(\mathcal{E}^1,d_1)$ of $r[\phi_s]_t$ as $s\to 2$. In this case since $\{r[\phi_s]_t\}_t$ is $s$-decreasing we have that $r[\phi_s]_t$ decreases to $r_t$ as $s\to 2$.  If $\mathcal S(X,\theta)$ is indeed complete, then $r_t = r[\psi]_t$ for some $[\psi] \in \mathcal S(X,\theta)$. But then we must have that  $r[\phi_s]_t \geq r_t = r[\psi]_t, \ t \geq 0$.  Letting $t \to \infty$, since $\mathcal C(\phi_s)=\phi_s$ (Lemma \ref{lem: P(tu)_model}), Proposition \ref{prop: ray_ceiling_prop}(i) implies that $\phi_s \geq \mathcal C(\psi) \geq \psi$ for any $s \in [1,2)$. In particular $\phi_s \geq \phi_2 \geq \mathcal C(\psi) \geq \psi$. As a result, Lemma \ref{lem: d_s_monotone} gives that $d_\mathcal S([\phi_s],[\phi_2]) \leq d_\mathcal S([\phi_s],[\psi]) \to 0$ as $s \to 2$, a contradiction with our above findings, hence $\mathcal S(X,\theta)$ is incomplete.

\section{The volume diamond inequality}\label{sect: volume diamond}

\begin{lemma}\label{lem: P(u,v)_mass_exist}
Assume that $u,v,w \in \textup{PSH}(X,\theta)$ are  such that $\int_X \theta_u^n +\int_X \theta_v^n>\int_X \theta_w^n$ and $\max(u,v) \leq w$. Then $P(u,v)\in \psh(X,\theta)$.
\end{lemma}

\begin{proof}   
We can assume without loss of generality that $u,v,w \leq 0$. Replacing $w$ with $P[\varepsilon V_\theta + (1-\varepsilon)w]$ for small enough $\varepsilon>0$ we can also assume that $\int_X \theta_w^n>0$ and $w= \mathcal{C}(w)$.

For $j \geq 0$ we set $u_j:=\max(u,w-j), v_j:=\max(v,w-j)$, $h_j:=P(u_j,v_j)$. Observe that $u_j, v_j, h_j$ have the same singularity type as $w$. We fix $s>0$ big enough, such that for all $j>s$, we have 
\[
\int_{\{u>w-s\}} \theta_{u_j}^n +\int_{\{v>w-s\}} \theta_{v_j}^n  = \int_{\{u>w-s\}} \theta_{u}^n +\int_{\{v>w-s\}} \theta_{v}^n >\int_X \theta_w^n, 
\]
where in the equality above we used Lemma \ref{lem: plurifine}.

It follows from  \cite[Lemma 3.7]{DDL2} and the above estimate that for $j>s$,
\begin{flalign*}
\int_{\{h_j\leq w-s\}} \theta_{h_j}^n \leq  \int_{\{u_j\leq w-s\}} \theta_{u_j}^n+\int_{\{v_j\leq w-s\}} \theta_{v_j}^n = 2 \int_X \theta_w^n -\int_{\{u>w-s\}} \theta_{u}^n -\int_{\{v>w-s\}} \theta_{v}^n 
<  \int_X \theta_w^n,
\end{flalign*}
where in the identity above we used the fact that $\{u_j\leq w-s\}= \{u\leq w-s\}$. Since $u_j,v_j$ decrease to $u,v$ respectively, it follows that $h_j \searrow P(u,v)$. We now rule out the possibility that $P(u,v)\equiv -\infty$. Indeed, suppose  $\sup_X h_j$ decreases to $-\infty$. From Lemma \ref{lem: compactness and model type} we obtain that  $\sup_X h_j = \sup_X (h_j - w) \searrow -\infty$. But then, for $j$ large enough the set $\{h_j\leq w-s\}$ coincides with $X$, contradicting our last integral estimate, since each $h_j$  has the same singularity type as $w$. 
\end{proof}

Plainly speaking, by the next lemma, the fixed point set of the map $\psi \to P[\psi]$ is stable under the operation $(\psi,\phi) \to P(\psi,\phi)$.

\begin{lemma}\label{lem: P_stab_fixed_point} Suppose $u_0,u_1 \in \textup{PSH}(X,\theta)$ are such that $P(u_0,u_1) \in \textup{PSH}(X,\theta)$, and $P[u_0]=u_0$ and $P[u_1]=u_1$. Then $P[P(u_0,u_1)]=P(u_0,u_1).$
\end{lemma}
\begin{proof} As $P(u_0,u_1)\leq \min(u_0,u_1) \leq 0$ and $P[P(u_0,u_1)] \leq P[u_0], P[u_1]$, it follows that
$$P(u_0,u_1) \leq P[P(u_0,u_1)] \leq P(P[u_0],P[u_1])=P(u_0,u_1).
$$
This shows that all the inequalities above are in fact equalities.
\end{proof}

\begin{prop}\label{prop: envelope mixed} Let $\phi,\psi \in \textup{PSH}(X,\theta)$ be such that $\phi = P[\phi]$, $\psi = P[\psi]$, and $P(\phi,\psi) \in \PSH(X,\theta)$. If $u \in \mathcal E(X,\theta,\phi)$, $v \in \mathcal E(X,\theta,\psi)$ and $\int_X \theta_{P(\phi,\psi)}^n>0$ then $P(u,v) \in \mathcal E(X,\theta,P(\phi,\psi))$. 
\end{prop}

\begin{proof} We can assume that $u\leq \phi$ and $v\leq \psi$. 

{\bf Step 1.}
We first prove that $P(u,\psi) \in \mathcal{E}(X,\theta,P(\phi,\psi))$. 
By assumption we  have $$\int_X \theta_u^n +\int_X \theta_{P(\phi,\psi)}^n >\int_X \theta_{\phi}^n,$$ and Lemma \ref{lem: P(u,v)_mass_exist} gives $P(u,\psi) = P(u,P(\phi,\psi)) \in \PSH(X,\theta)$. 
Fixing $b>1$, it follows from Lemma \ref{lem: subextension} that $u_b := P_{\theta}(bu -(b-1) \phi) \in \PSH(X,\theta)$. For $1<b<t$ we have 
$$
\phi\geq u_b \geq bt^{-1}u_t + (1-bt^{-1}) \phi. 
$$
Comparing the total mass via \cite{WN19} and letting $t\to +\infty$ we see that $u_b \in \mathcal E(X,\theta,\phi)$. The previous argument then gives $P(u_b,\psi) \in \PSH(X,\theta)$. On the other hand we also have 
$$
u \geq b^{-1} u_b + (1-b^{-1})\phi, 
$$
therefore $P(u,\psi) \geq b^{-1} P(u_b,\psi) + (1-b^{-1})P(\phi,\psi)$. Comparing the total mass via \cite{WN19} and letting $b \to +\infty$ we arrive at $\int_X \theta_{P(u,\psi)}^n \geq \int_X \theta_{P(\phi,\psi)}^n$, hence the conclusion.

{\bf Step 2. } We prove that $P(u,v) \in \PSH(X,\theta)$. 
It follows from \cite{WN19}, the assumption $v \in \mathcal{E}(X,\theta,\psi)$, and the first step that 
$$
\int_X \theta_{P(u,\psi)}^n + \int_X \theta_v^n= \int_X \theta_{P(\phi,\psi)}^n + \int_X \theta_{\psi}^n > \int_X \theta_{\psi}^n.
$$ 
Since $\max(P(u,\psi),v) \leq \psi$,  Lemma \ref{lem: P(u,v)_mass_exist} can be applied giving $P(u,v) = P(P(u,\psi),v) \in \PSH(X,\theta)$. 

{\bf Step 3.} We conclude the proof. Fixing $b>1$, it follows from Lemma \ref{lem: subextension} that  $v_b:= P_{\theta}(bv -(b-1)\psi) \in \PSH(X,\theta)$. For $1<b<t$ we have 
$$
\psi\geq v_b \geq bt^{-1}v_t + (1-bt^{-1}) \psi. 
$$
Comparing the total mass via \cite{WN19} and letting $t\to +\infty$ we see that $v_b \in \mathcal E(X,\theta,\psi)$. By the second step we have that $P(u,v_b) \in \PSH(X,\theta)$.  On the other hand we also have 
$$
v \geq b^{-1} v_b + (1-b^{-1})\psi,
$$
therefore $P(u,v) \geq b^{-1} P(u,v_b) + (1-b^{-1})P(u,\psi)$. Comparing the total mass via \cite{WN19} and letting $b \to +\infty$ we arrive at $\int_X \theta_{P(u,v)}^n \geq \int_X \theta_{P(u,\psi)}^n$. Combining this and the first step we arrive at the conclusion.
\end{proof}

Finally, we prove the main result of this section:

\begin{theorem}\label{thm: max and envelope}
Let $u,v\in \textup{PSH}(X,\theta)$ and assume that $P(u,v) \in \PSH(X,\theta)$. Then
\begin{equation}\label{eq: volume_diamond_ineq}
\int_X \theta_u^n + \int_X \theta_v^n \leq \int_X \theta_{\max(u,v)}^n +\int_X \theta^n_{P(u,v)}.
\end{equation} 
\end{theorem}

\begin{proof} 
It is enough to check \eqref{eq: volume_diamond_ineq} only in the case when $u = P[\phi]$ and $v = P[\psi]$ for some $\phi,\psi \in \textup{PSH}(X,\theta)$. Indeed, we first note that, for each $t>0$,  $\max(P(u+t,0),P(v+t,0))$ and  $\max(u,v)$ have the same singularity type. Since  $\max(P(u+t,0),P(v+t,0))$ increases a.e. to $\max(P[u],P[v])$, 
a direct application of \cite[Theorem 2.3, Remark 2.5]{DDL2} gives that 
$\int_X \theta_u^n = \int_X \theta_{P[u]}^n, \ \int_X \theta_v^n = \int_X \theta_{P[v]}^n \ \textup{ and } \ \int_X \theta_{\max(u,v)}^n = \int_X \theta_{\max(P[u],P[v])}^n.$ 
If $\theta_{P(P[u],P[v])}^n>0$ then  Proposition \ref{prop: envelope mixed} above ensures that  $\int_X \theta_{P(u,v)}^n = \int_X \theta_{P(P[u],P[v])}^n$, while in the zero mass case, the equality   follows from \cite{WN19}. 

For the rest of the argument we assume that $u=P[\phi]$ and $v = P[\psi]$, and for convenience we introduce $w := \max(u,v) \leq 0$. 
For $t>0$ we set  $u_t := \max(u,w-t)$ and $v_t := \max(v,w-t)$. Observe that, by \cite[Theorem 3.8]{DDL2},  $\theta_u^n$ is supported on $\{u=0\}=\{u_t=0\}$ which is contained in $\{u>w-t\}$, for $t>0$. From this  and   plurifine locality, Lemma \ref{lem: plurifine} we have  
\begin{equation}
    \label{eq: diamond c1}
\theta_{u_t}^n = \mathbbm{1}_{\{u>w-t\}} \theta_u^n + \mathbbm{1}_{\{u\leq w-t\}} \theta_{u_t}^n =  \mathbbm{1}_{\{u_t=0\}} \theta_u^n + \mathbbm{1}_{\{u\leq w-t\}} \theta_{u_t}^n= \theta_u^n +  \mathbbm{1}_{\{u\leq w-t\}} \theta_{u_t}^n.
\end{equation}
By the same argument applied for $v_t$ we have
\begin{equation}
    \label{eq: diamond c2}
\theta_{v_t}^n = \mathbbm{1}_{\{v>w-t\}} \theta_v^n + \mathbbm{1}_{\{v\leq w-t\}} \theta_{v_t}^n =  \mathbbm{1}_{\{v_t=0\}} \theta_v^n + \mathbbm{1}_{\{v\leq w-t\}} \theta_{v_t}^n= \theta_v^n +  \mathbbm{1}_{\{v\leq w-t\}} \theta_{v_t}^n. 
\end{equation}
Integrating over $X$ and noting that, by \cite{WN19}, $\int_X \theta_{u_t}^n =\int_X \theta_{v_t}^n=\int_X \theta_w^n$, we obtain 
\begin{equation}
    \label{eq: diamond c3}
    \int_X \theta_w^n -\int_X \theta_u^n = \int_{\{u\leq w-t\}} \theta_{u_t}^n, \ \ \int_X \theta_w^n -\int_X \theta_v^n = \int_{\{v\leq w-t\}} \theta_{v_t}^n, \ t>0. 
\end{equation}
Building on \eqref{eq: diamond c1}, \eqref{eq: diamond c2},  an application of  \cite[Lemma 3.7]{DDL2} gives 
\begin{flalign}
\theta_{P(u_t,v_t)}^n &\leq \mathbbm{1}_{\{P(u_t,v_t)=u_t\}}\theta_{u_t}^n + \mathbbm{1}_{\{P(u_t,v_t)=v_t\}}\theta_{v_t}^n \nonumber \\
&\leq \big(\mathbbm{1}_{\{P(u_t,v_t)=u_t=0\}} \theta_u^n+\mathbbm{1}_{\{P(u_t,v_t)=v_t=0\}} \theta_v^n\big)  + \mathbbm{1}_{\{u \leq w-t\}}\theta_{u_t}^n + \mathbbm{1}_{\{v \leq w-t\}}\theta_{v_t}^n. \label{eq: P_meas_est}
\end{flalign}
In particular, $\theta_{P(u_t,v_t)}^n$ is supported on the union of the disjoint sets  $A_t := \{u \leq w-t\} \cup \{v \leq w-t\}$ and $\{P(u_t,v_t)=0 \}$.  From here, since $P(u_t,v_t)$ has the same singularity type as $w$, an integration allows to conclude that:
\begin{flalign*}
\int_{\{P(u_t,v_t)=0\}} \theta_{P(u_t,v_t)}^n =  \int_X \theta_w ^n - \int_{A_t} \theta_{P(u_t,v_t)}^n 
\geq \int_X \theta_w ^n - \int_{\{u \leq w-t\}} \theta_{u_t}^n - \int_{\{v \leq w-t\}} \theta_{v_t}^n.
\end{flalign*}
where in the inequality  we have used  \eqref{eq: P_meas_est}. 
Now, using the above inequality, \eqref{eq: diamond c3}, and Proposition \ref{prop: MA dec maxi}  we let $t\to +\infty$ to get
$$\int_X \theta_{P(u,v)}^n \geq \int_{\{P(u,v)=0\}} \theta_{P(u,v)}^n\geq -\int_X \theta_w ^n + \int_X \theta_u^n +\int_X \theta_v^n,$$
finishing the proof. 
\end{proof}

\begin{remark}\label{rem: no equality in diamond}
If $\textup{dim }X=1$ then we have actually equality in \eqref{eq: volume_diamond_ineq}. Indeed, since $
\frac{\max(u,v)+P(u,v)}{2} \leq \frac{u+v}{2}$, an application of the main result of \cite{WN19}, yields the equality in \eqref{eq: volume_diamond_ineq}.
On the other hand, equality can not hold in general. Consider $X=\mathbb{C}\mathbb{P}^2$ with $\omega:=\omega_{FS}$ the Fubini Study metric and we view $(z_1,z_2)\in\mathbb{C}^2$ as a chart of $\mathbb{C}\mathbb{P}^2$. Let $\rho$ be the local potential of $\omega_{FS}$.  Set 
$$
u(z_1,z_2) : = \log (|z_1|^2+ |z_2|^2) - \rho \ ; \ v(z_1,z_2): = \log (|z_1|^2 +|z_2-1|^2) -\rho, \ w(z_1,z_2) := \log |z_1|^2 -\rho.
$$
Then  $w\leq P(u,v)$, hence $ P(u,v)$ is a $\omega_{FS}$-psh function and $\int_X \omega_u^2 =\int_X \omega_v^2=\int_X \omega_{P(u,v)}^2 =0$. On the other side, $\max(u,v)$ 
is bounded, hence $\int_X \omega_{\max(u,v)}^2 = \int_X \omega^2>0$. 
\end{remark}

As a consequence of \eqref{eq: volume_diamond_ineq} we show that every $d_\mathcal S$-convergent sequence in $\mathcal S_\delta(X,\theta)$ has a subsequence that can be sandwiched between an increasing and a decreasing $d_\mathcal S$-convergent sequence:

\begin{theorem}\label{thm: conv_subs_monotone} Let $\delta >0$, and suppose that $[u_j],[u] \in \mathcal S_\delta(X,\theta)$ satisfies  $d_\mathcal S([u_j],[u]) \to 0$ and $u_j = P[u_j]$, $u = P[u]$. Then there exist a subsequence $u_{j_k}$ and  decreasing/increasing sequences $v_{j_k},w_{j_k} \in \textup{PSH}(X,\theta)$ such that $u_{j_k} \leq v_{j_k} \searrow u$, $u_{j_k} \geq w_{j_k} \nearrow u$, and  $d_\mathcal S([v_{j_k}],[u]) \to 0$, $d_\mathcal S([w_{j_k}],[u]) \to 0$. 
\end{theorem}

As we will see below, for the appropriate subsequence $u_{j_k}$, the potentials $v_{j_k}$ and $w_{j_k}$ can be chosen as follows:

\begin{equation}\label{eq: v_w_j_k_formula}
w_{j_k}:= P(u_{j_k},u_{j_{k+1}},\ldots) \ \textup{ and } \ v_{j_k}:= \textup{usc} \Big(\sup_{l \geq k} u_{j_l}\Big).
\end{equation}

\begin{proof} We will pass to subsequences multiple times during the proof. Doing this, we can begin to assume that $d_\mathcal S([u_j],[u_{j+1}]) \leq \frac{1}{C^{2j}}$, where $C>1$ is the constant from Proposition \ref{prop: poor_Pythagorean_S}. To start, we introduce the following decreasing sequence:
$$v_j:= \textup{usc}\Big( \sup_{k \geq j} u_k\Big).$$ 
Trivially, $u_k \leq v_j$ for all $k \geq j$, and it follows from Proposition \ref{prop: Cauchy_monotone} and its proof (see \eqref{eq: v_j_u_j_C}) that $d_\mathcal S([u],[v_j]) \to 0$. 

Now we construct the sequence $w_j$. After possibly taking another subsequence, we can assume that $d_\mathcal S([u],[u_{j}])\leq \frac{1}{(n+1)C2^j}$ and $ d_\mathcal S([u],[v_{j}]) \leq \frac{1}{(n+1)C2^j}$. Lemma \ref{lem: mixed_MA_dC_conv} then implies
\begin{equation*}
\bigg|\int_X \theta_{u_j}^n - \int_X \theta_{u}^n \bigg| \leq \frac{1}{2^j} \ \ \textup{ and } \ \ \bigg|\int_X \theta_{v_j}^n - \int_X \theta_{u}^n \bigg| \leq \frac{1}{2^j},
\end{equation*}
hence
\begin{equation} \label{eq: sandwich 1}
    \bigg|\int_X \theta_{u_{j+k}}^n - \int_X \theta_{v_{j+k-1}}^n \bigg| \leq \frac{1}{2^{j+k}}+  \frac{1}{2^{j+k-1} }< \frac{1}{2^{j+k-2}}. 
\end{equation}

Fix $j_0>0$ large enough so that $2^{-j_0+3} < \delta$. We claim that, for all $j>j_0, k \geq 0$, we have 
$
P(u_j,...,u_{j+k}) \in \PSH(X,\theta),
$
and 
\begin{equation}\label{eq: P_estij}
\int_X \theta_{u_j}^n - \sum_{l=0}^k\frac{1}{2^{j+l-2}}\leq \int_X \theta^n_{P(u_j,\ldots, u_{j+k})} \leq \int_X \theta_{u_j}^n.
\end{equation}
We argue inductively assuming the claim for $k-1$, as the case $k=0$ is obvious. From \eqref{eq: sandwich 1},\eqref{eq: P_estij} we have 
\begin{eqnarray}\label{masses induction}
 \int_X \theta^n_{P(u_j,u_{j+1}, \ldots, u_{j+k-1})} +\int_X \theta_{u_{j+k}}^n   &>& \int_X \theta_{u_j}^n - \sum_{l=0}^{k-1}\frac{1}{2^{j+l-2}} + \int_X \theta_{v_{j+k-1}}^n - \frac{1}{2^{j+k-2}}\\
\nonumber &> & \int_X \theta_{u_j}^n - 2^{-j+3} +\int_X \theta_{v_{j+k-1}}^n > \int_X \theta_{v_{j+k-1}}^n,
\end{eqnarray}
where the last inequality follows from the choice of $j_0$. Since $P(u_j,u_{j+1}, \ldots u_{j+k-1}),  {u_{j+k}} \leq v_{j+k-1}$, it then follows from Lemma \ref{lem: P(u,v)_mass_exist} that $P(P(u_j,...,u_{j+k-1}), u_{j+k})=P(u_j,...,u_{j+k}) \in \PSH(X,\theta)$.  
We next apply Theorem \ref{thm: max and envelope} to obtain
\begin{flalign*}
\int_X \theta_{u_{j+k}}^n + \int_X \theta^n_{P(u_j,u_{j+1}, \ldots, u_{j+k-1})} & \leq \int_X \theta_{\max(P(u_j, \ldots, u_{j+k-1}), u_{j+k})}^n + \int_X \theta^n_{P(u_j,u_{j+1}, \ldots, u_{j+k})} \\
& \leq \int_X \theta_{v_{j+k-1}}^n + \int_X \theta^n_{P(u_j,u_{j+1}, \ldots, u_{j+k})}, 
\end{flalign*}
where  in the second inequality we used \cite[Theorem 1.1]{DDL2}. The claim follows from \eqref{masses induction} and the above inequality. 

Set $w_j^k:= P(u_j,\ldots, u_{j+k})$.  It follows from Lemma \ref{lem: P_stab_fixed_point} that $w_j^k=P[w_j^k]$, hence $\sup_X w_j^k=0$. Therefore, the decreasing limit $\lim_k w_j^k$ is a $\theta$-psh function (it is not identically $-\infty$). 
 Proposition \ref{prop: MA dec maxi} now gives that $\int_X \theta_{w_j^k}^n \to \int_X \theta_{w_j}^n$. Putting this together with \eqref{eq: P_estij} we obtain that  
$$0<\int_X \theta_{u_j}^n -\frac{1}{2^{j-3}}\leq \int_X \theta^n_{w_j} \leq \int_X \theta_{u_j}^n, \ j >j_0.$$
Moreover from Corollary \ref{coro: C stable increasing} together with Proposition \ref{prop: C_operator_properties}(i) we know that $\mathcal{C}(w_j)=P[w_j]= w_j$. 
Let $w \in \textup{PSH}(X,\theta)$ be the increasing limit of $\{w_j\}_j$, and  $v$ be the decreasing limit of $\mathcal{C}(v_j)=P[v_j]$. It follows from Corollary \ref{coro: C stable increasing} that $w=\mathcal{C}(w)$ and $v= \mathcal{C}(v)$.  Since $d_{\mathcal{S}}([v_j],[u]) \to 0$ and $d_{\mathcal{S}}([v_j],[v]) \to 0$ (where the last assertion follows from Proposition \ref{prop: mixed MA dec maxi}) we have that $d_{\mathcal{S}}([u],[v])=0$, hence $u=v$ since they are both model potentials (Theorem \ref{criteria d_1 vanishing}(iii)). 

 By monotonicity of $\{w_j\}_j$ and $\{v_j\}_j$ we see that $w_j \leq w \leq P[v_j], \ j \geq 0$. From here and the above integral estimate we have that $\int_X \theta_w^n = \int_X \theta_u^n$ and $w\leq u$, hence $u=w$ since they are both model potentials.  
 
 Finally, according to Lemma \ref{lem: mon_limit_complete},  $\lim_j d_\mathcal S([w_j],[u])=\lim_j d_\mathcal S([w_j],[w])=0$, finishing the proof.\end{proof}

\section{Semicontinuity of multiplier ideal sheaves} \label{sect: ideal sheaves}

For $[u] \in \mathcal S(X,\theta)$ we denote by $\mathcal J[u]$  the multiplier ideal sheaf associated to the singularity type $[u]$. Recall that $\mathcal J[u]$ is the sheaf of germs of holomorphic
functions $f$ such that $|f|^2 e^{-u}$ is locally integrable on $X$. We now give a version of a theorem of Guan and Zhou \cite{GZh15,GZh16} adapted to our context:

\begin{theorem}\label{thm: mult_ideal_semicont} Let $[u],[u_j] \in \mathcal S(X,\theta)$ be such that $d_\mathcal S([u_j],[u]) \to 0$.  Then there exists $j_0 \geq 0$ such that $\mathcal J[u] \subseteq \mathcal J[u_j]$ for all $j \geq j_0$.
\end{theorem}

\begin{proof} We first assume that there exists $\delta>0$ such that $u_j,u \in \mathcal{S}_{\delta}(X,\theta)$,  for all $j\geq 0$.

We note that we can assume that $P[u_j]=u_j$ and $P[u]=u$. Indeed, since $P[u_j]$ is the increasing limit of the potentials $P(0,u_j + c)$ and $[P(0,u_j + c)] = [u_j]$ for any $c \in \Bbb R$, it follows from \cite{GZh15,GZh16} (see \cite[Theorem 0.8]{Dem15} for a survey) that $\mathcal J[u_j] = \mathcal J[P[u_j]]$. Similarly, $\mathcal J[u]= \mathcal J[P[u]]$.

By contradiction let us assume that $\mathcal J[u]$ is not a subsheaf of $\mathcal{J}[u_j]$ for big enough $j$. Then there exists a subsequence of $[u_j]$, again denoted by $[u_j]$, such that 
\begin{equation}\label{eq: J_non_inclusion}
\mathcal J[u] \not \subseteq \mathcal J[u_j], \ \ j\geq 0.
\end{equation}

After possibly taking another subsequence, via Theorem \ref{thm: conv_subs_monotone}, we can further assume that  
there exists $\{w_j\}_j \subset \textup{PSH}(X,\theta)$ increasing such that $w_j \leq u_j$ and $w_j  \nearrow u$.  
Using \cite[Theorem 0.8]{Dem15} again, it follows that $\mathcal J[u] = \mathcal J[w_j] \subseteq \mathcal J[u_j]$ for all $j$ greater than some fixed $j_0$. But this is a contradiction with our assumption \eqref{eq: J_non_inclusion}.

We now treat the general case. Using \cite[Theorem 0.8]{Dem15} we can find $\varepsilon>0$ small enough such that $\mathcal{J}[u]=\mathcal{J}[u+ \varepsilon V_{\theta}]$.   By Lemma \ref{lem: from omega to two omega} below, $d_{\mathcal{S},(1+\varepsilon)\theta}([u_j+\varepsilon V_{\theta}],[u+\varepsilon V_{\theta}]) \to 0$. Thus, by the first step we have that $\mathcal{J}[u]=\mathcal{J}[u+\varepsilon V_{\theta}] \subset \mathcal{J}[u_j+\varepsilon V_{\theta}]$, for $j\geq j_0$, where $j_0$ depends on $\varepsilon$. But $\mathcal{J}[u_j+\varepsilon V_{\theta}] \subset \mathcal{J}[u_j]$, hence the conclusion. 
\end{proof}

\begin{lemma}\label{lem: from omega to two omega}
	For $\varepsilon \in [0,1]$ there exists $C=C(n)>1$ such that for all $u,v \in \PSH(X,\theta)$ we have 
	$$
	\frac{1}{C}d_{\mathcal{S}, \theta} ([u], [v]) \leq d_{\mathcal{S},(1+\varepsilon)\theta} ([u+\varepsilon V_{\theta}],[v+\varepsilon V_{\theta}]) \leq   Cd_{\mathcal{S},\theta} ([u], [v]).
	$$
\end{lemma}

\begin{proof} Let us assume that $u \leq v$.   The general case reduces to this particular situation using  Proposition \ref{prop: poor_Pythagorean_S}. Set $u_{\varepsilon}:= u + \varepsilon V_{\theta}$, $v_{\varepsilon}:= v + \varepsilon V_{\theta}$.
Then Lemma \ref{lem: d_s_monotone} gives the following:
\begin{flalign} 
& d_{\mathcal{S}}([u], [v])= \frac{1}{(n+1)} \sum_{j=0}^n \left( \int_X \theta_{v}^j \wedge \theta_{V_{\theta}}^{n-j} - \int_X \theta_{u}^j \wedge \theta_{V_{\theta}}^{n-j}\right ),\label{eq: sheaf 1} \\
 d_{\mathcal{S},(1+\varepsilon)\theta}([u_{\varepsilon}], [v_{\varepsilon}])
& = \frac{1}{n+1} \sum_{j=0}^n \left( \int_X (\theta_{v} + \varepsilon \theta_{V_{\theta}})^j \wedge (1+\varepsilon)^{n-j}\theta_{V_{\theta}}^{n-j} -\int_X (\theta_{u} + \varepsilon \theta_{V_{\theta}})^j \wedge (1+\varepsilon)^{n-j}\theta_{V_{\theta}}^{n-j}\right ) \nonumber \\
& = \frac{1}{n+1} \sum_{j=0}^n  (1+c_j)\left( \int_X \theta_{v}^j \wedge \theta_{V_{\theta}}^{n-j} - \int_X \theta_{u}^j \wedge \theta_{V_{\theta}}^{n-j}\right )\label{eq: sheaf 2}
\end{flalign}
where $0\leq c_j = O(\varepsilon), j=0,...,n$ are positive constants depending only on $n,\varepsilon$.  From \eqref{eq: sheaf 1}, \eqref{eq: sheaf 2}, and the fact that $\int_X \theta_u^j \wedge \theta_{V_{\theta}}^{n-j} \leq \int_X \theta_v^j \wedge \theta_{V_{\theta}}^{n-j}$ (which follows from \cite[Theorem 2.4]{DDL2}) we obtain the desired estimate. 
\end{proof}

\section{Stability of solutions to CMAE with prescribed singularity type}\label{sect: stability}

In this section we show that solutions to a family of complex Monge-Amp\`ere equations with varying singularity type converge as governed by the $d_\mathcal S$-topology:

\begin{theorem} Given $\delta >0$ and $p >1$ suppose that:\\
$\circ$ $[\phi_j],[\phi] \in \mathcal S_\delta(X,\theta), \ j \geq 0$ satisfy $\phi_j = P[\phi_j]$,  $\phi = P[\phi]$ and $d_\mathcal S([\phi_j],[\phi]) \to 0$. \\
$\circ$ $f_j,f \geq 0$ are such that $\| f\|_{L^p},\| f_j\|_{L^p}$, $p>1$, are uniformly bounded  and $f_j \to_{L^1}f$.\\
$\circ$ $\psi_j,\psi \in \textup{PSH}(X,\theta), \ j \geq 0$ satisfy $\sup_X \psi_j=0$, $\sup_X \psi=0$ and 
$$
\begin{cases}
\theta_{\psi_j}^n = f_j \omega^n\\
[\psi_j]=[\phi_j] \  \ \ \ 
\end{cases}  
, \ \ \  
\begin{cases}
\theta_{\psi}^n = f \omega^n\\
[\psi]=[\phi]
\end{cases}.
$$
Then $\psi_j$ converges to $\psi$ in capacity, in particular $\| \psi - \psi_j\|_{L^1} \to 0$.
\end{theorem}
\begin{proof} 
 First we claim that it is enough to show that any subsequence of $\psi_j$ contains a subsequence that converges in capacity to $\psi$.
Indeed, suppose that $\psi_j$ does not converge to $\psi$ in capacity. Then there exists $\varepsilon >0$ such that $\limsup_j \textup{Cap}_\omega(\{ |\psi_j - \psi|> \varepsilon\})> \delta$ for some $\delta >0$. In particular, there exists $j_k \to \infty$ such that $\textup{Cap}_\omega(\{ |\psi_{j_k} - \psi|> \varepsilon\})> \delta$ for all $j_k$. But then $\{\psi_{j_k}\}_k$ would contain a subsequence converging to $\psi$ in capacity, giving a contradiction.

We take a subsequence of $f_j$, again denoted by $f_j$, such that $\|f_j-f_{j+1} \|_{L^1} \leq \frac{1}{2^{j+1}}, \ j \geq 0$. By an elementary argument $g := f_0 + \sum_{j \geq 0} |f_{j+1} - f_j| \in L^1(\omega^n)$ and $f_j,f \leq g$ for all $j \geq 0$.

Now let us take a subsequence of $\phi_j$, again denoted by $\phi_j$, such that there exists $w_j,v_j \in \textup{PSH}(X,\theta)$ increasing/decreasing sequences with $w_j \leq \phi_j \leq v_j$ such that  $d_\mathcal S([w_j],[\phi]) \to 0$ and $d_\mathcal S([v_j],[\phi]) \to 0$. This is possible due to Theorem \ref{thm: conv_subs_monotone}. Moreover we recall that $v_j =\textup{usc}\big( \sup_{k \geq j} \phi_k\big)$, and $w_j$ arises as the decreasing limit $w_j := \lim_{k} w_j^k$, where $w_j^k := P(\phi_j,\phi_{j+1},\ldots,\phi_{j+k})$ (see \eqref{eq: v_w_j_k_formula}).

We consider $\gamma_j := \textup{usc}\big(\sup_{k \geq j} \psi_k \big)\geq\psi_j$. Observe that $\sup_X \gamma_j=0$, $j\geq 0$. For this sequence  \cite[Lemma 4.27]{DDL2} gives that $\theta_{\gamma_j}^n \geq \big(\inf_{k \geq j} f_k\big) \omega^n$. 

Since $[\phi_j],[\phi] \in \mathcal S_\delta(X,\theta)$,  \cite[Theorem 4.7]{DDL4} gives existence of $C>0$ such that
\begin{equation}\label{ineq sandwich}
\phi-C \leq \psi \leq \phi \ \ \textup{ and } \ \ \phi_j-C \leq \psi_j \leq \phi_j, \ \ j \geq 0.
\end{equation}
In particular, we have that $w_j - C \leq \gamma_j \leq v_j$ for all $j \geq 0$. Hence the  monotonicity of the sequences $w_j$ and $v_j$ implies that
$$w_j - C \leq \gamma:= \lim_k \gamma_k \leq v_j, \ j \geq 0.$$
Letting $j \to \infty$, we obtain that 
$$\int_X \theta_\phi^n= \lim_{j\to \infty} \int_X \theta_{w_j}^n \leq  \int_X \theta_\gamma^n \leq \lim_{j\to \infty} \int_X \theta_{v_j}^n=\int_X \theta_\phi^n.$$
Consequently,  the conditions of \cite[Theorem 2.3]{DDL2} hold for the decreasing sequence $\{ \gamma_j\}_j$, yielding the estimate $\theta_\gamma^n \geq f \omega^n$. By comparing total masses again, we conclude that in fact $\theta_\gamma^n = f \omega^n.$ By uniqueness of solutions in $\mathcal E(X,\theta,\phi)$ (\cite[Theorem 4.29]{DDL2}), and noting that $\sup_X \gamma =\sup_X \psi =0$, we obtain that $\gamma = \psi$.

This also shows that $\psi_j$ converges in $L^1$ (and a.e.) to $\psi$. Indeed since $\sup_X \psi_j =0$ we can assume that, up to extracting, $\psi_j$ converges to some $\psi_{\infty}$ in $L^1$ and a.e.. Then (by construction) $\gamma_j$ also does converge to $\psi_{\infty}$. But the limit of $\gamma_j$ is $\gamma=\psi$. 

We fix $r \in (1,p)$. By our assumptions on the $f_j,f$,  and the H\"older inequality we obtain that $f_j \to f$ in $L^r$. 
Let $s>1$ be the conjugate exponent of $r$, i.e. $1/s +1/r =1$.  Take $\varepsilon>0$ so small that $e^{-\varepsilon \psi_j} \to e^{-\varepsilon \psi}$ in $L^{s}(X,\omega^n)$ and consequently $\sup_j \|e^{-\varepsilon \psi_j}\|_{L^s(X,\omega^n)} <+\infty$. 
This is possible as we explain below. For $x= -\varepsilon\psi_j, y =-\varepsilon\psi$ we have that $x,y\geq 0$ and an elementary  argument gives 
$$
|e^x -e^y|^s \leq e^{s(x+y)} |x-y|^s. 
$$
Thus, after applying H\"older's inequality twice, we obtain
\begin{flalign*}
\int_X |e^{-\varepsilon\psi_j} -e^{-\varepsilon\psi}|^s \omega^n & \leq \varepsilon^s  \int_X e^{-s\varepsilon(\psi_j+\psi)} |\psi_j-\psi|^s \leq \varepsilon^s  \left(\int_X e^{-2s\varepsilon(\psi_j+\psi)}  \omega^n\right)^{1/2} \left(\int_X |\psi_j-\psi|^{2s} \omega^n\right)^{1/2} \\
&\leq  \varepsilon^s  \left(\int_X e^{-4s\varepsilon\psi_j}  \omega^n\right)^{1/4} \left(\int_X e^{-4s\varepsilon\psi}  \omega^n\right)^{1/4} \left(\int_X |\psi_j-\psi|^{2s} \omega^n\right)^{1/2}. 
\end{flalign*}
The convergence statement for $e^{-\varepsilon\psi_j}$ then follows because $\psi_j$ converges to $\psi$ in any $L^t$, $t>1$, while, since $\sup_X \psi_j= \sup_X \psi=0$, Skoda's uniform theorem (\cite{Zer01},\cite[Theorem 2.50]{GZ17}) ensures that both $e^{-4s \varepsilon \psi_j}$ and $e^{-4s \varepsilon \psi}$ are uniformly bounded in $L^1$ for $\varepsilon>0$ small enough.

Now set $h_j:= e^{-\varepsilon \psi_j} f_j, h:= e^{-\varepsilon \psi} f$. We have
$$
\int_X |h_j-h| \omega^n \leq \int_X e^{-\varepsilon \psi_j} |f_j-f| \omega^n +  \int_X |e^{-\varepsilon \psi_j}-e^{-\varepsilon \psi}| f \omega^n. 
$$
Applying H\"older's inequality with exponents $r$ and $s$ we conclude that $\|h_j-h\|_{L^{1}} \to 0$. 
Up to extracting again we can assume that $h_j,h \leq \tilde{g}$ where $\tilde{g}\in L^1(X,\omega^n)$ is constructed exactly as the function $g$ at the beginning of the proof. 

From \eqref{ineq sandwich} we have
$$w_j^k - C \leq \chi_j^k := P(\psi_j,\psi_{j+1}, \ldots, \psi_{j+k}) \leq w^k_j,$$
giving that $\chi_j^k$ is a  $\theta$-psh function. Observe then that the Monge-Amp\`ere equation for $\psi_j$ rewrites as
$\theta_{\psi_j}^n = e^{\varepsilon \psi_j} h_j \omega^n.
$ Thus, Lemma \ref{lem: supersolution} below gives
$$
\theta_{\chi_j^k}^n \leq e^{\varepsilon \chi_j^k}\big(\sup_{l \geq j} h_l\big)  \omega^n. 
$$
From the first statement of \cite[Theorem 2.3]{DDL2} we have 
$$\theta_{\chi_j}^n \leq \liminf_k \theta_{\chi_j^k}^n \leq  e^{\varepsilon \chi_j}  \big(\sup_{l \geq j} h_l\big)  \omega^n, $$
where $\chi_j := \lim_k \searrow \chi_j^k$. Also $w_j - C \leq \chi_j  \leq w_j$.
Now we argue that the increasing limit $\chi:= \lim_j \chi_j = \psi$. Indeed, we can apply \cite[Theorem 2.3]{DDL2} and the dominated convergence theorem to conclude that 
\begin{equation}\label{eq: supersolution stability}
\theta_\chi^n \leq e^{\varepsilon \chi} h \,\omega^n = e^{\varepsilon(\chi -\psi)} f\omega^n.
\end{equation}
On the other hand, \cite[Theorem 1.1 and Theorem 2.3]{DDL2} together with Lemma \ref{lem: mixed_MA_dC_conv} give that
$$\int_X \theta_\chi^n = \lim_j \int_X \theta_{\chi_j}^n = \lim_j \int_X \theta_{w_j}^n = \int_X \theta_\phi^n=\int_X f \omega^n,$$
hence $\chi \in \mathcal{E}(X,\theta, \phi)$. Recall that we also have $\psi \in \mathcal{E}(X,\theta,\phi)$. 
By the comparison principle, \cite[Corollary 3.16]{DDL2}, and \eqref{eq: supersolution stability}, we have 
$$
\int_{\{\chi<\psi\}} \theta_\psi^n \leq \int_{\{\chi<\psi\}} \theta_\chi^n \leq \int_{\{\chi<\psi\}}  e^{\varepsilon(\chi- \psi)} f\omega^n = \int_{\{\chi<\psi\}}  e^{\varepsilon(\chi- \psi)} \theta_{\psi}^n \leq \int_{\{\chi<\psi\}}  \theta_{\psi}^n.
$$
It then follows that  all the above inequalities become equalities,  and $\theta_{\psi}^n(\chi<\psi)=0$. Therefore, all  terms in the above are zero. In particular $\theta_{\chi}^n(\chi<\psi)=0$, and by the domination principle, \cite[Proposition 3.11]{DDL2}, we have that $\chi \geq \psi$.

On the other side, by construction of $\chi_j$ and $\gamma_j$ we have that $\chi_j \leq \gamma_j$, and so $\chi \leq \gamma =\psi$ finally giving $\chi =\psi$.

To summarize, we proved existence of two monotone sequences $\chi_j,\gamma_j$ such that $\chi_j \leq \psi_j \leq \gamma_j$ with $\gamma_j$ decreasing to $\psi$ and $\chi_j$ increasing to $\psi$. In particular $\chi_j$ and $\gamma_j$ converge in capacity to $\psi$ (\cite[Proposition 4.25]{GZ17}). This implies that $\psi_j$ converges to $\psi$ in capacity, finishing the proof.
\end{proof}

\begin{lemma}\label{lem: supersolution}
Assume that $u,v,P(u,v) \in \PSH(X,\theta)$, and $\mu$ is a positive non pluripolar measure, $\varepsilon>0$, $0\leq f,g \in L^1(\mu)$. If $\theta_u^n \leq e^{\varepsilon u} f\mu$, $\theta_v^n \leq e^{\varepsilon v} g\mu$, then 
$$
\theta_{P(u,v)}^n \leq e^{\varepsilon P(u,v)} \max(f,g) \mu. 
$$
\end{lemma}
\begin{proof}

By replacing $\mu$ with $\id_{X\setminus P}\mu$, where $P:= \{u=v=-\infty\}$, we can assume that $\mu(P)=0$. Since $\mu(X)<+\infty$, the function $r \to  \mu(\{u \leq v+r \})$ is monotone increasing. Such functions have at most a countable number of discontinuities, hence for almost every $r \geq 0$ we have that $\mu(\{u=v+r\})=0$. 
For such $r$ we set $\varphi_r:=P_{\theta}(u,v+r)$, and note that $\varphi_r \searrow P_\theta(u,v)$ as $r\rightarrow 0$.  
It then follows from \cite[Lemma 3.7]{DDL2} that we can write 
\begin{flalign*}
	\theta_{\varphi_r}^n & \leq \mathbbm{1}_{\{\varphi_r=u\}} \theta_u^n+\mathbbm{1}_{\{\varphi_r=v+r\}}\theta_v^n  \leq  \mathbbm{1}_{\{\varphi_r=u\}} e^{\varepsilon u} f \mu  +\mathbbm{1}_{\{\varphi_r=v+r\}} e^{\varepsilon v} g \mu\\
	& \leq  \mathbbm{1}_{\{\varphi_r=u\}} e^{\varepsilon \varphi_r} \max(f,g) \mu  +\mathbbm{1}_{\{\varphi_r=v+r\}} e^{\varepsilon \varphi_r } \max(f,g) \mu \leq e^{\varepsilon \varphi_r} \max(f,g) \mu,
	\end{flalign*}
 where in the last inequality we used the fact that $\mu(\{u=v+r\})=0$. Letting $r\searrow 0$, we use \cite[Theorem 2.3]{DDL2} to arrive at the conclusion. 
\end{proof}

\let\omegaLDthebibliography\thebibliography 
\renewcommand\thebibliography[1]
{
  \omegaLDthebibliography{#1}
  \setlength{\parskip}{1pt}
  \setlength{\itemsep}{1pt plus 0.3ex}
}

\noindent{\sc University of Maryland}\\
{\tt tdarvas@math.umd.edu}\vspace{0.1in}\\
\noindent{\sc Sorbonne Universit\'e}\\
{\tt eleonora.dinezza@imj-prg.fr}\vspace{0.1in}\\
\noindent {\sc Universit\'e Paris-Sud}\\
{\tt hoang-chinh.lu@u-psud.fr}
\end{document}